\documentclass[10pt]{article}

\usepackage{amssymb,mathtools,enumitem,mathrsfs,amsthm} 
\usepackage{overpic,subcaption,hyperref}
\usepackage[linesnumbered,ruled,vlined]{algorithm2e}
\usepackage{todonotes}\reversemarginpar
\newcounter{CONT}

\usepackage{geometry}

\definecolor{mygreen}{rgb}{0, 0.75, 0}
\definecolor{myred}{rgb}{0.784, 0, 0}

\newtheorem{theorem}{Theorem}
\newtheorem{prop}{Proposition}
\newtheorem{lemma}{Lemma}
\newtheorem{corollary}{Corollary}
\newtheorem{conjecture}{Conjecture}

\newtheorem{remark}{Remark}
\newtheorem{definition}{Definition}

\newcommand{\claimbegin}[1] {\par\noindent\underline{Proof of #1:}}
\newcommand{\claimend}[1] {\hfill\underline{End proof of #1.}\par\vspace*{2mm}}

\usepackage{xspace}

\newcommand{\Touch} {\text{Touch}}

\newcommand{\Max} {\text{Max}}
\newcommand{\MAX} {\textbf{Max}\xspace}
\newcommand{\TOUCH} {\textbf{Touch}\xspace}

\newcommand{\FF} {\texttt{FourForests}\xspace}
\newcommand{\FTF} {\texttt{FifteenForests}\xspace}
\newcommand{\RC} {\texttt{RightLabel}\xspace}
\newcommand{\LC} {\texttt{LeftLabel}\xspace}
\newcommand{\SC} {\texttt{SpecialLabel}\xspace}
\newcommand{\CT} {\texttt{LabelTouch}\xspace}
\newcommand{\CST} {\texttt{LabelSpecialTouch}\xspace}
\newcommand{\TM} {\texttt{TripleMax}\xspace}
\newcommand{\TSM} {\texttt{TripleSpecialMax}\xspace}
\newcommand{\se} {edge}
\newcommand{\scc} {special}
\newcommand{\NCS} {\text{NCSP}\xspace}
\newcommand{\BNCS} {\text{BNCSP}\xspace}
\newcommand{\FCN} {Path Covering with Forests Number\xspace}
\newcommand{\PCFN} {\textrm{PCFN}}

\newcommand{\NULL} {\texttt{NULL}\xspace}

\begin{document}
\pagestyle{plain}

\title{Non-Crossing Shortest Paths are Covered with Exactly\\ Four Forests}

\author{Lorenzo Balzotti\footnote{Dipartimento di Scienze di Base e Applicate per l’Ingegneria, Sapienza Universit\`a di Roma, Via Antonio Scarpa, 16, 00161 Roma, Italy. \texttt{lorenzo.balzotti@sbai.uniroma1.it}.}}
%\and
%{Paolo G. Franciosa\footnote{Dipartimento di Scienze Statistiche, Sapienza Universit\`a di Roma, p.le Aldo Moro 5, 00185 Roma, Italy. \texttt{paolo.franciosa@uniroma1.it}}}}

\date{}	
\maketitle

\begin{abstract} 
Given a set of paths $P$ we define the \emph{\FCN of $P$} (\PCFN($P$)) as the minimum size of a set $F$ of forests satisfying that every path in $P$ is contained in at least one forest in $F$. We show that \PCFN($P$) is treatable when $P$ is a set of non-crossing shortest paths in a plane graph or subclasses. We prove that if $P$ is a set of non-crossing shortest paths of a planar graph $G$ whose extremal vertices lie on the same face of $G$, then $\PCFN(P)\leq 4$, and this bound is tight.
\end{abstract}

\texttt{Keywords}: shortest paths, planar undirected graphs, non-crossing paths, \FCN

\section{Introduction}

In this article we investigate the structure of particular sets of paths. Given a set of paths $P$ we define the \emph{\FCN of $P$} (\PCFN($P$)) as the minimum size of a set $F$ of forests satisfying that every path in $P$ is contained in at least one forest in $F$. A trivial upper bound is $\PCFN(P)\leq |P|$, in which every forest is composed exactly by one path.

We note that if there are different subpaths joining the same pair of vertices, then it may happen $\PCFN(P)=|P|$. This cannot happen if $P$ is a set of shortest paths, such that there is a unique shortest path for any pair of vertices. We deal with a slightly more general case for which we require the \emph{single-touch} property: given two paths in $P$, their intersection is still a path. If all paths in $P$ are shortest paths in $G$, then the single-touch property is implied by ensuring the uniqueness of the shortest path in $G$, that can be obtained through a tiny perturbation of edges' weights (of $G$).

%We show that \PCFN($P$) is interesting/treatable when $P$ is a set of non-crossing shortest paths in a planar graph or related classes. We prove that if $P$ is a set of non-crossing shortest paths of a planar graph $G$ whose extremal vertices lie on the same face of $G$, then the $FCV(P)$ is 4. We also give a linear time algorithm for (forests in) $F$.

There is a very restricted literature dealing with this problem for general graph, a first recent result by Bodwin~\cite{bodwin} in 2019 develops a structural theory of unique shortest paths in real weighted graphs: the author characterizes exactly which sets of node sequences can be realized as unique shortest paths in a graph with arbitrary real edge weights. The characterizations are based on a new connection between shortest paths and topology; in particular, the new forbidden patterns are in natural correspondence with two-colored topological 2-manifolds, which are visualized as polyhedra. Even if finding shortest paths is a classical problem with applications in several fields, there are not other results focused on shortest paths' structure.

We prove that if $P$ is a set of non-crossing shortest paths in a plane (i.e., a planar graph with a fixed embedding) undirected graph $G$, whose extremal vertices lie on the same face of $G$, then $\PCFN(P)\leq4$, and this bound is tight, where two paths are \emph{non-crossing} if they do not cross each other in the given embedding.

We first explain that this setting, i.e., non-crossing paths in plane graphs, is not too restrictive. Removing the non-crossing property makes $\PCFN(P)$ dependent on the dimension of $P$. We briefly prove this with an example. In Figure~\ref{fig:6_paths} there are six pairwise crossing paths in a grid graph having the extremal vertices on the same face of $G$.  Note that any set of three paths forms a cycle, hence, each forest can contain at most two paths. So the \FCN of these six paths is three. It is trivial to generalize this example to a set $P$ of single-touch shortest paths in a plane graph whose extremal vertices lie on the same face so that $\PCFN(P)=|P|/2$.

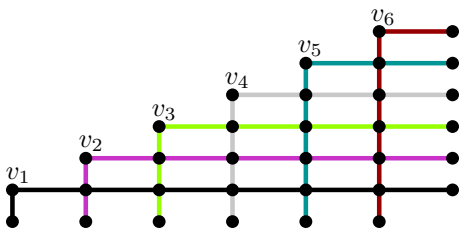
\begin{figure}[h]
\captionsetup[subfigure]{justification=centering}
\centering
%FIGURA 1
	\begin{subfigure}{6cm}
\begin{overpic}[width=6cm,percent]{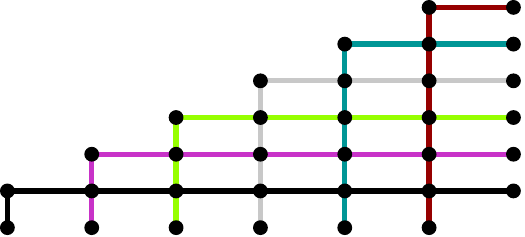}
\put(0,11){$v_1$}
\put(16,18){$v_2$}
\put(32,25){$v_3$}
\put(48,32){$v_4$}
\put(64,39){$v_5$}
\put(80,46){$v_6$}
\end{overpic}
%\caption{}\label{LABEL1}
\end{subfigure}
\caption{a set $P$ of pairwise crossing paths satisfying $\PCFN(P)=|P|/2$. If edges incident on $v_i$, for $i\in[6]$, have weight less than 1 and other edges have weight 1, then colored paths are unique shortest paths between their extremal vertices.} 
\label{fig:6_paths}
\end{figure}

Our result about the \FCN of non-crossing shortest paths in plane undirected graphs answers to the open problem by Balzotti and Franciosa~\cite{err_giappo}, asking for an algorithm able to list a path in a time proportional to its length. Indeed, all algorithms that find non-crossing shortest paths in plane graphs~\cite{balzotti-franciosa_2,steiger,giappo2} actually find their union, and so listing a path is not trivial. In this way, we can apply the result by Gabow and Tarjan~\cite{gabow-tarjan} about lowest common ancestor queries, hence it is possible to list each path in $P$ in time proportional to its length. This application was the idea behind the introduction of the \FCN.

Finding non-crossing shortest paths in a plane graph is a problem with primary applications in VLSI layout~\cite{bhatt-leighton,leighton1,leighton2}, and thanks to the article by Reif~\cite{reif} it is also used to compute the max flow in undirected plane graphs~\cite{hassin-johnson,italiano} and vitality problems~\cite{balzotti-franciosa_3}. The above cited articles~\cite{giappo2,steiger} solve this problem for positive weighted graphs , while in~\cite{balzotti-franciosa_2} a linear time algorithm is shown for the unweighted case. In this settings, the extremal vertices of the non-crossing shortest path are always on the same face of the planar embedding (in~\cite{giappo2} it is also studied the case in which the extremal vertices are on two faces), while in~\cite{erickson-nayyeri} the extremal vertices are on $h$ face boundaries. It is stated in~\cite{erickson-nayyeri} that the union of a set $P$ of non-crossing shortest paths in a plane graph whose extremal vertices lie on the same face can be covered with at most two forests so that each path is contained in at least one forest, i.e., $\PCFN(P)=2$. We stress that this result is incorrect, a first counterexample is shown in Figure~\ref{fig:3_forests} whose \FCN is 3 (it can be proved by a simple enumeration).

\begin{figure}[h]
\centering
\begin{overpic}[width=10cm,percent]{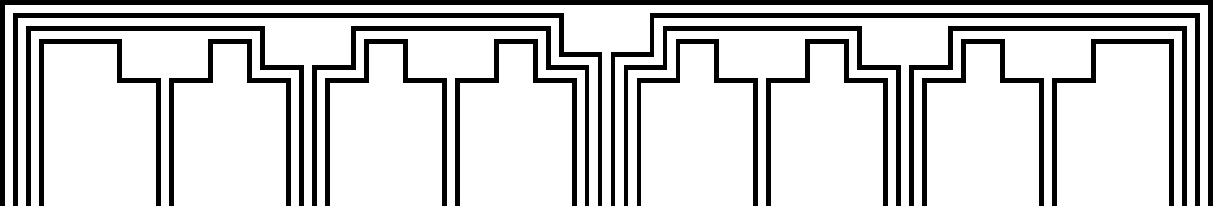}

\end{overpic}
\caption{a set of 15 non-crossing paths whose \FCN is 3 (parallel adjacent segments represent overlapping paths).} 
\label{fig:3_forests}
\end{figure}

\paragraph{Related problems} The \FCN is strictly linked to the concept of arboricity. The \emph{arboricity} of an undirected graph $G$ is the minimum number of forests $\gamma_f(G)$ into which its edges can be partitioned. It measures how a graph is dense, indeed, graphs with many edges have high arboricity, and graphs with high arboricity must have a dense subgraph. By the well-known Nash-Williams Theorem~\cite{nash-williams}  (proved also independently by Tutte~\cite{tutte})
\begin{equation*}
\gamma_f(G)=\max_{X\subseteq V(G)}\left\lceil{\frac{|E(G[X])|}{|X|-1}}\right\rceil
\end{equation*}
where $G[X]$ denotes the subgraph of $G$ induced by $X$. The \emph{fractional arboricity} was introduced by Payan in~\cite{payan}, see also~\cite{catlin-grossman,goncalves-francese}. Arboricity is studied for general graphs and it is specialized for planar graph and subclasses of planar graphs. By the above cited Nash-Williams Theorem~\cite{nash-williams}, every planar graph has arboricity 3, i.e., every planar graph can be covered with at most 3 forests, and if it has girth greater or equal to 4, then it decomposes into two forests. In~\cite{goncalves,he-hou} planar graphs with girth larger than some constant are decomposed into a forest and a graph with bounded degree. Decomposition of planar graphs into a forest and a matching has been studied in~\cite{bassa-burns,borodin-ivanova,borodin-kostochka,he-hou,montassier-de_mendez,wang-zhang}.

Arboricity is one of the many faces of \emph{graphs covering}~\cite{beineke,harary1,harary2,orlin-covering} which is a classical problem in graph theory. A recent and complete overview about covering problems can be found in~\cite{schwartz}. The classical covering problem asks for covering an input graph $H$ with graphs from a fixed covering class $\mathcal{G}$. Some variants of the problems are in~\cite{knauer-ueckerdt}. In the arboricity problem the family $\mathcal{G}$ consists in forests. Other kinds of arboricity have been introduced in literature, as star arboricity~\cite{akiyama-kano,algor-alon,alon-mcdiarmid}, caterpillar arboricity~\cite{goncalves-caterpillar,goncalves-ochem}, linear arboricity~\cite{akiyama-exoo,alon,rautenbach-volkann,wu-linear}, pseudo arboricity~\cite{hakimi,picard-queyranne-forest} in which graphs are covered with star forests, caterpillar forests, linear forests, and pseudoforests (undirected graphs in which every connected component has at most one cycle), respectively.

If the covering class is the class of planar graph, outerplanar graph, or interval graph, then we deal with planar thickness~\cite{beineke}, outerplanar thickness~\cite{mutzel-odenthal} or track number~\cite{gyarfas}, respectively.

\paragraph{Future works} With the introduction of the \FCN, we propose an original nuance on the classical covering problem and we would like to deal with the \FCN in a more general context. The first generalization is to ask whether the \FCN of a set of non-crossing single-touch shortest paths in a plane graph is bounded by a constant, even if the extremal vertices of the paths are not required to lie on the same face. We conjecture that such a constant exists for general plane graph thanks also to the following remark, whose proof is omitted.

\begin{remark}
Let $P$ be a set of non-crossing single-touch shortest paths in a plane graph $G$ and let $v$ be a vertex of $G$. Then all paths of $P$ containing $v$ form a tree.
\end{remark}

\begin{conjecture}
There exists $\ell\in\mathbb{N}$ such that $\PCFN(P)\leq\ell$, for any set $P$ of non-crossing single-touch shortest paths in a plane graph.
\end{conjecture}

The \FCN can be studied also for a set of paths $P$ beyond planar graphs, as in $k$-planar graphs~\cite{pach-toth}, $k$-quasi-planar graphs~\cite{agarwal-aronov}, RAC graphs~\cite{didimo-eades}, fan-crossing-free graphs~\cite{cheong-peled}, fan-planar graphs~\cite{kaufmann-ueckert}, $k$-gap-planar graphs~\cite{bae-baffier}; a complete survey about these graph classes can be found in~\cite{hong-tokuyama}. %This problem is interesting also in subclasses\todob{citazioni a survey se si trovano!} of planar graphs, in which the shortest paths can be chosen with the extremal vertices lying everywhere on the graph or on $h$ face boundaries.

It would be interesting to investigate the \FCN in the case of crossing paths in plane graphs or related graph classes. As shown in Figure~\ref{fig:6_paths}, for a set of crossing paths $P$ its \FCN may depend on $|P|$. To deal with these cases it is necessary to understand ``how much'' the paths cross each others. In~\cite{schaefer} several notions about crossing, crossing number and variants are explained, thus the \FCN can be studied with respect to these measures. We note that in~\cite{schaefer} the notions are given for graphs, but they can be extended to sets of paths.

Clearly, the \FCN can be generalized to other covering families as done for the arboricity or the classical covering problem. Thus we can introduce, for example, the Path Covering with Stars Number, the Path Covering with Caterpillars Number, Path Covering with Planars Number and so on.

\paragraph{Our approach} In a first step we organize all the paths into a partial order named \emph{genealogy tree}. Then we translate the \FCN into a problem of \emph{forest labeling}, i.e., a labeling assigning labels to paths such that then union of paths with a same label is a forest. The number of distinct labels corresponds to the upper bound of \FCN. A crucial result is that we can establish whether a labeling is a forest labeling by restricting the check only to the faces of the graph resulting from the union of the paths. At this point the main result is reached by three steps: first we prove that the \FCN is constant by introducing a simple algorithm \FTF able to give a forest labeling which uses at most 15 labels (see Theorem~\ref{th:15_forest} and Corollary~\ref{cor:15_forest_teorico}); then we refine this algorithm obtaining algorithm \FF able to give a forest labeling which uses at most 4 labels; finally we show that this result is tight by exhibiting a set of non-crossing shortest paths in a plane graph whose \FCN is at least 4 (we recall that in Figure~\ref{fig:3_forests} there is a set of non-crossing shortest paths $P$ such that $\PCFN(P)=3$). This proves our main result.

Now we briefly explain the strategy behind our algorithms. We recursively decompose all the paths with respect to the genealogy tree and intersections between paths. In this way, at each iteration our algorithms have to assign labels only to paths intersecting a fixed path $p$.

\paragraph*{Structure of the article}

In Section~\ref{sec:preliminaries} preliminaries notations and definitions are given, we also propose a labeling approach for our problem. In Section~\ref{sec:restricting_to_faces} we prove that we can restrict us only to $U$'s faces, where $U$ is the graph obtained by the union of all non-crossing shortest paths. In Section~\ref{sec:15_ALL}  we prove that the \FCN of a set of non-crossing shortest paths in a plane graph is at most 15. In Section~\ref{sec:4_ALL} we decrease this bound to 4. Finally, in Section~\ref{sec:controesempio} we show that the latter bound is tight. Conclusions are in Section~\ref{sec:conclusions}.

\section{Preliminaries}\label{sec:preliminaries}

In this section we introduce some definitions and notations. In Subsection~\ref{sub:notations} we give general notation. In Subsection~\ref{sub:the_problem} we formally define the problem and we describe our labeling approach. In Subsection~\ref{sub:genealogy_tree} we partially order the non-crossing shortest paths by the \emph{genealogy tree}.

\subsection{Notations}\label{sub:notations}

All graphs in this article are undirected. Let $G=(V,E)$ be a graph, where $V$ is a set of \emph{vertices} and $E$ is a collection of pairs of vertices called \emph{edges}. We recall standard union and intersection operators on graphs.

\begin{definition}
Given two graphs $G=(V(G),E(G))$ and $H=(V(H),E(H))$, we define the following operations and relations:
\begin{itemize}
\item $G\cup H=(V(G)\cup V(H),E(G)\cup E(H))$,
\item $G\cap H=(V(G)\cap V(H),E(G)\cap E(H))$,
\item $H\subseteq G\Longleftrightarrow V(H)\subseteq V(G)$ and $E(H)\subseteq E(G)$,
\item $G\setminus H=(V(G),E(G)\setminus E(H))$.
\end{itemize}
\end{definition}

Given a graph $G=(V(G),E(G))$, given an edge $e$ and a vertex $v$ we write, for short, $e\in G$ in place of $e\in E(G)$ and $v\in G$ in place of $v\in V(G)$.

We denote by $uv$ the edge whose endpoints are $u$ and $v$. We use angle brackets to denote ordered sets. For example, $\{a,b,c\}=\{c,a,b\}$ and $( a,b,c)\neq ( c,a,b)$. Moreover, for every $\ell\in\mathbb{N}$ we denote by $[\ell]$ the set $\{1,\ldots,\ell\}$.

Given a path $p\in P$ and two vertices $u,v$ of $p$, we define $p[u,v]$ the subpath of $p$ from $u$ to $v$. We say that a path $p$ is an \emph{$ab$ path} if its extremal vertices are $a$ and $b$. For an $ab$ path $q$ and a $bc$ path $r$, we define $q\circ r$ as the (possibly not simple) path obtained by the  concatenation of $q$ and $r$.

We denote by $f^\infty_G$ the external face of a plane graph $G$, if no confusions arise we remove the subscript $G$.

For a simple cycle $C$ of a plane graph $G$, we define the \emph{region bounded by $C$} the maximal subgraph of $G$ whose external face has $C$ as boundary.

\subsection{The problem and a labeling approach}\label{sub:the_problem}

In this subsection we give a formally definition of our problem and we introduce a labeling approach. From now on $G$ denotes a plane graph and $f^\infty$ denotes the external face of $G$. For convenience, we assume that the extremal vertices of paths in $P$ lie on $f^\infty$.

\begin{definition}
Two paths $p$ and $q$ are \emph{single-touch} if their intersection is still a (possibly empty) path.
\end{definition}

We observe that the single-touch property can be always required for a set of shortest paths, and it also known as \emph{consistent property} in the literature of path systems~\cite{bodwin}. We stress that in this article we use the single-touch property rather than the property of being shortest paths. Indeed, it is easy to describe a set $P$ of $k$ non-crossing shortest paths in a plane graphs whose \FCN is $k-1$ if the single-touch property is not required.

\begin{definition}\label{definition:NCS}
Given a set of paths $P$ we say that $P$ is a \emph{set of non-crossing shortest paths (\NCS)} if there exists a plane graph $G$ such that %\todob{il primo punto sembra inutile visto il terzo}
\begin{itemize}\itemsep0em
%\item every path in $P$ is contained in $G$,
\item for each $p\in P$, the extremal vertices of $p$ are in $f^\infty$,
\item for each $p\in P$, $p$ is a shortest path in $G$ and $P$ is a set of single-touch paths,
\item for each $p,q\in P$, $p$ and $q$ are non-crossing in $G$.
\end{itemize}
\end{definition}

We observe that the single-touch property can be always required for a set of shortest paths, and it also known as \emph{consistent property} in the literature of path systems~\cite{bodwin}. Indeed, the single-touch property is implied by ensuring the uniqueness of the shortest path in $G$, that can be obtained through a tiny perturbation of edges' weights. We stress that in this chapter we use the single-touch property rather than the property of being shortest paths. Indeed, it is easy to describe a set $P$ of $k$ non-crossing shortest paths in a plane graphs whose \FCN is $k-1$ if the single-touch property is not required.

From now on, if no confusions arise, given a \NCS $P$, $G$ is the plane graph in Definition~\ref{definition:NCS}. We study the \FCN of a \NCS $P$ by using a labeling function that assigns labels to paths in $P$. So we say that a function $\mathcal{L}:Q\mapsto [k]$ is a \emph{path labeling of $P$} if $\mathcal{L}$ assigns one value of $[k]$ to each path in $Q$, for some $k\in\mathbb{N}$ and $Q\subseteq P$. If no confusion arises, then we omit the dependence on $P$.

\begin{definition}
Given a \NCS $P$, given a path labeling $\mathcal{L}$ of $P$, $\mathcal{L}:P\mapsto [k]$, we say that $\mathcal{L}$ is a \emph{forest labeling} if $\bigcup_{\{p\in P\,|\, \mathcal{L}(p)=i\}}p$ is a forest for each $i\in[k]$.
\end{definition}

We extend the definition of labeling to a set of paths $Q$ by denoting $\mathcal{L}(Q)=\bigcup_{q\in Q}\mathcal{L}(q)$. Moreover, given an edge $e$, we define $\mathcal{L}(e)=\mathcal{L}(\{q\in P \,|\, e\in E(q)\})$, hence, $\mathcal{L}(e)$ may contain more colors.

For a path $p$, we denote its extremal vertices by $x_p$ and $y_p$. W.l.o.g., we assume that the terminal pairs are distinct, i.e., there is no pair $p,q\in P$ such that $\{x_p,y_p\}=\{x_q,y_q\}$. Let $\gamma_p$ be the path in $f^\infty$ that goes clockwise from $x_p$ to $y_p$, for $p\in P$. We say that pairs $\{(x_p,y_p)\}_{p\in P}$ are \emph{well-formed} if for all $q,r\in P$ either ${\gamma_q}\subseteq{\gamma_r}$ or ${\gamma_r}\supseteq{\gamma_q}$ or ${\gamma_q}$ and ${\gamma_r}$ share no edges. We note that if terminal pairs are well-formed, then there always exists a set of pairwise non-crossing shortest paths, each one joining a pair. The revers is not true if some paths are subpaths of the infinite face of $G$; this case is not interesting in the applications and it has never been studied in literature, where the terminal pairs are always assumed to be well-formed. Hence we assume that pairs $\{(x_p,y_p)\}_{p\in P}$ are well-formed.% this assumption can be easily obtained by requiring all the $x_p$'s and the $y_p$'s to be vertices with degree one one the infinite face.

Finally, we observe that the non-crossing and single-touch properties of the paths imply that the embedding of the paths is unique; this fact is formally proved in~\cite{err_giappo}.

\subsection{Genealogy tree}\label{sub:genealogy_tree}

Given a \NCS $P$, we define here a partial ordering as in~\cite{giappo2} that represents the inclusion relation between the $\gamma_p$'s. This relation intuitively corresponds to an \emph{adjacency} relation between non-crossing shortest paths joining each pair.

Choose an arbitrary $p_1\in P$ such that there are neither $x_q$ nor $y_q$, with $q\neq p_1$ and $q\in P$, walking on $f^{\infty}$ from $x_{p_1}$ to $y_{p_1}$ (either clockwise or counterclockwise), and let $e_1$ be an arbitrary edge on that walk. For each $q \in P$, we can assume that $e_1\not\in\gamma_q$, indeed if it is not true, then it suffices to switch $x_q$ with $y_q$. Given $p,q\in P$, We say that $p \preceq q$ if $\gamma_p\subseteq\gamma_q$. We define the \emph{genealogy tree}  $T_g^P$ of $P$ as the transitive reduction of poset $(P,\preceq)$; if no confusion arises, then we omit the apex $P$. We consider $T_g^P$ as a rooted tree.

If $p\preceq q$, then we say that $p$ is a \emph{descendant} of $q$ and $q$ is an \emph{ancestor} of $p$. Given $p,q\in P$, we say that $p$ and $q$ are \emph{uncomparable} if $p\npreceq q$ and $q\npreceq p$. %Finally, we denote by $p(q)$ the \emph{parent of $q$}, i.e.,  the lowest ancestor of $q$ with respect to $\preceq$.\todob{notazione infelice $p(q)$ ma forse non si usa}

Figure~\ref{fig:gt} shows the extremal vertices of a \NCS $P$, and the corresponding genealogy tree $T^P_g$.  From now on, in all figures we draw $f^\infty_G$ by a solid light grey line.

Every path $p$ splits $G$ into two subgraphs $Int_p$, the ``internal'' subgraph of $G$ with respect to $p$, and $Ext_p$, the ``external'' subgraph of $G$ with respect to $p$; in order to well define these subgraphs, we require that $p\subseteq Int_{p_1}$ and $Int_p\subseteq Int_{p_1}$, for every $p\in P$. We stress that $p= Int_p\cap Ext_p$, for every $p\in P$.

\begin{figure}[h]
\captionsetup[subfigure]{justification=centering}
\centering
%FIGURA 1
	\begin{subfigure}{4cm}
\begin{overpic}[width=6cm,percent]{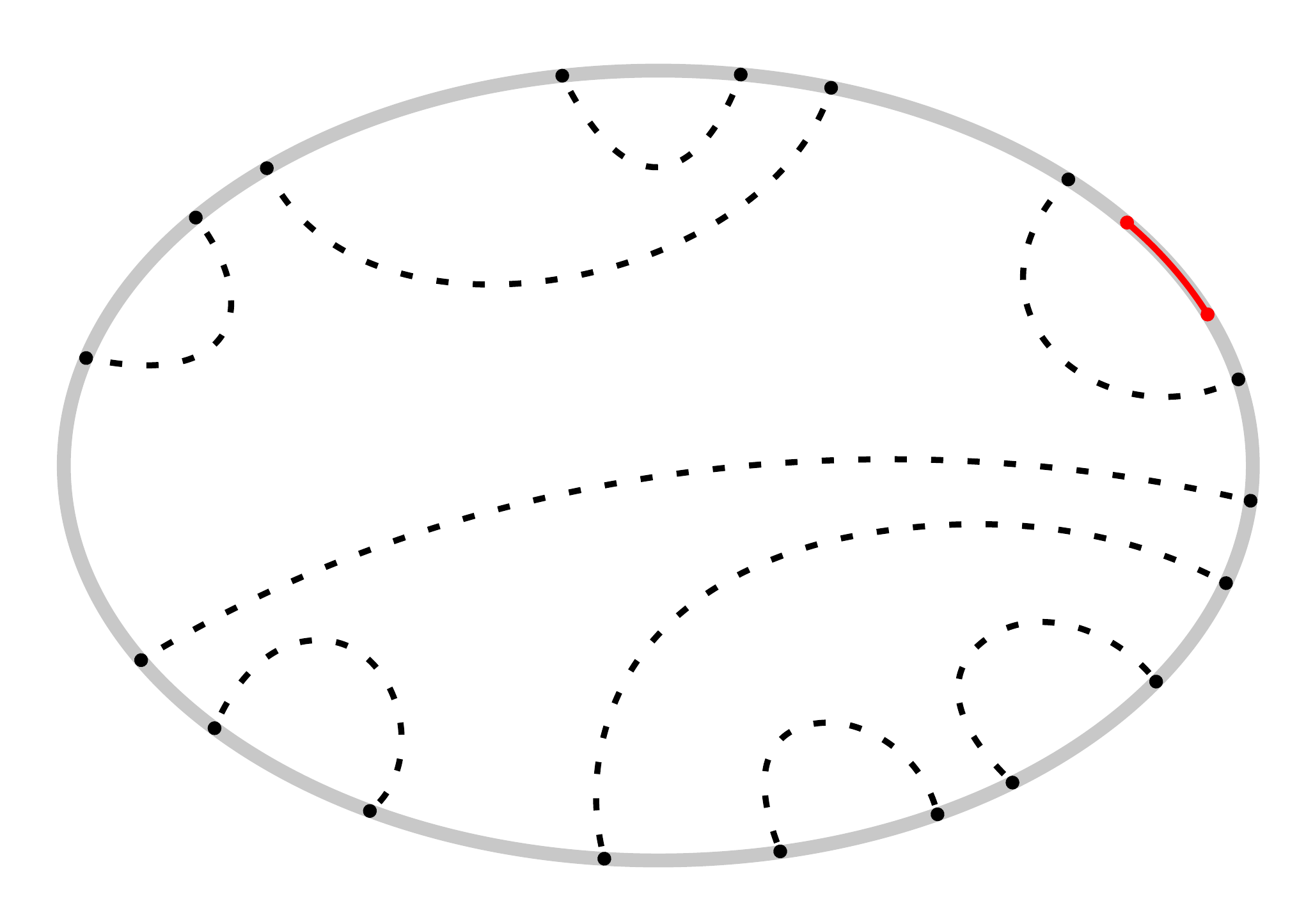}
\put(91,51){\color{red} $e_1$}
\put(-4.5,32){$f^\infty$}

\put(-1,43.5){$x_7$}

\put(64,66.5){$y_8$}

\put(82.5,59){$y_1$}
\put(96,42){$x_1$}
\put(97.5,30.5){$x_2$}
\put(5.5,16.8){$y_2$}
\put(95.4,24){$x_3$}
\put(44.5,0.3){$y_3$}
\put(89.9,16.5){$x_4$}
\put(79,7){$y_4$}
\put(70.2,3.7){$x_5$}
\put(60,0.9){$y_5$}
\put(25,4){$x_6$}
\put(14,9.5){$y_6$}
%\put(-2,44){$x_7$}
\put(8,55){$y_7$}
\put(15,60.5){$x_8$}
\put(38.5,67.8){$x_9$}
\put(55,68.3){$y_9$}

\end{overpic}
%\caption{}\label{LABEL1}
\end{subfigure}
\qquad\qquad\qquad\qquad\quad
%FIGURA 2
	\begin{subfigure}{4.7cm}
\begin{overpic}[width=4.7cm,percent]{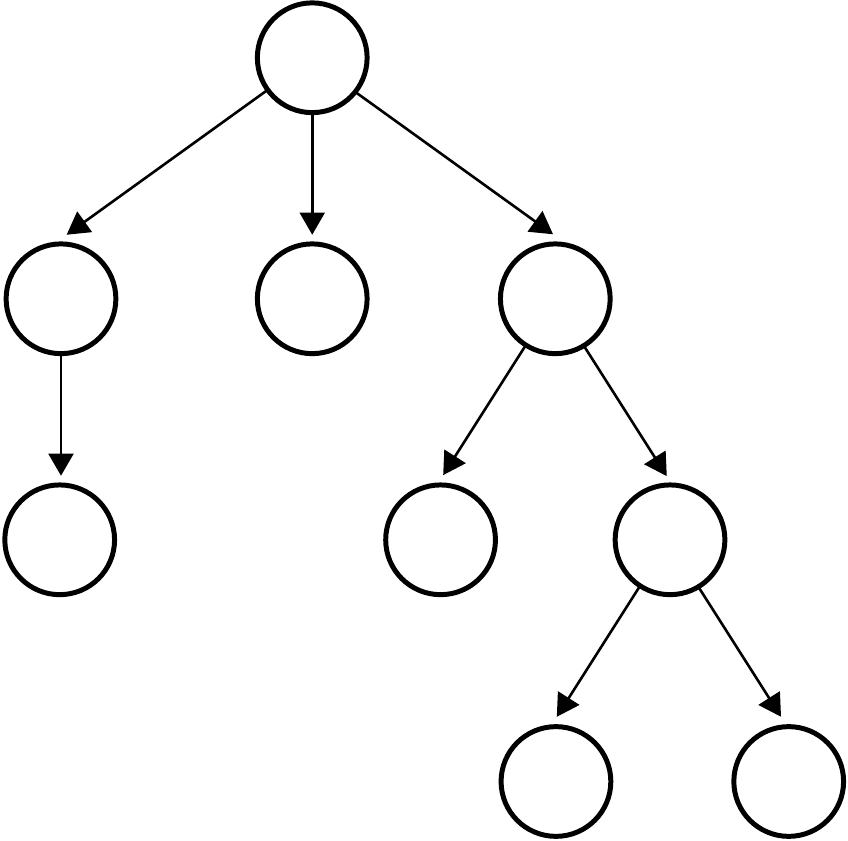}
\put(33.7,91){$p_1$}
\put(3.9,62.7){$p_8$}
\put(33.7,62.7){$p_7$}
\put(62.1,62.7){$p_2$}
\put(3.9,34.2){$p_9$}
\put(48.8,34.2){$p_6$}
%\put(60,48.3){$p_4$}
\put(75.9,34.2){$p_3$}
\put(62.5,5.7){$p_5$}
\put(90,5.7){$p_4$}
\end{overpic}
%\caption{}\label{LABEL2}
\end{subfigure}
\caption{on the left the extremal vertices of paths in $P$, on the right its genealogy tree $T_g$.}
\label{fig:gt}
\end{figure}
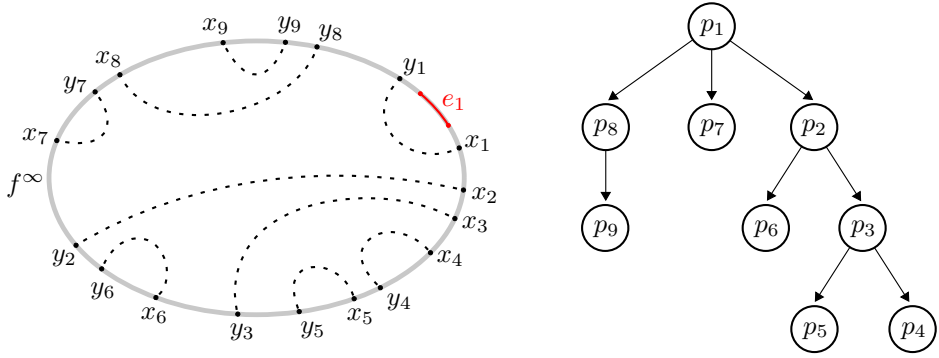

\section{Restricting to faces}\label{sec:restricting_to_faces}

Our goal is to find a path labeling such that every set of paths covering a cycle does not share the same label. In this section we prove that we can restrict our check to $U$'s faces, where $U$ is the graph obtained by the union of all paths in $P$. In a first step we prove that given a \NCS $P$ its genealogy tree can be binarized (see Subsection~\ref{sub:binarization}). After this simplification, we can establish whether a path labeling is a forest labeling by checking how it works on the faces of $U$ (see Subsection~\ref{sec:solving_faces} and in particular Theorem~\ref{th:faces}).

\subsection{Binarization of the genealogy tree}\label{sub:binarization}

In order to simplify the treatment, given a \NCS $P$, we can assume that its genealogy tree $T_g$ is a binary tree in the following way. For two paths $p,q\in P$, we say that $p\lhd q$ if $x_1,x_q,y_q,x_p,y_p,y_1$ appear in this order on $f^\infty$; we recall that 1 is the root of the genealogy tree. Given $p\in P$ having $r$ children with $r\geq3$, we order its set of children $( q_1,\ldots,q_r)$, so that $q_j\lhd q_{j+1}$ for $j\in[r-1]$. If we add a terminal pair $( x_{p'},y_{p'})$ so that $x_{p'}=c_2$ and $y_{p'}=y_r$, then $p$ has only two children $q_1$ and $p'$. By repeating this procedure, we obtain the binarization of $T_g$. Note that the number of terminal pairs becomes at most doubles. 

We observe that, being $U$ connected, then there exists an $x_{p'}y_{p'}$ path that does not cross other paths. Moreover, this path is a shortest path in $U$ but it might be not a shortest path in $G$. We do not care about this because it is an auxiliary path. By repeating this reasoning for all $i\in[k]$ having more than two children, we can assume that $T_g$ is binary. Being $U$ connected, then the binarization can be obtained in $O(|P|)$ time because we have only to add some terminal pairs.

\begin{definition}\label{def:binary_set_of_paths}
Given a set of paths $P$ we say that $P$ is a \emph{binary set of non-crossing shortest paths (\BNCS)} if 
\begin{itemize}\itemsep0em
\item $P$ is a \NCS,
\item each path $p\in P$ has zero or two children in the genealogy tree,
\item for each $p\in P$, if $p$ has two children in the genealogy, then $p,p_r,p_\ell$ form a face.
\end{itemize}
\end{definition}

We observe that the last requirement in Definition~\ref{def:binary_set_of_paths} can be always obtained by modifying $f^\infty$ contracting vertices and changing edge lengths in order to maintain the shortest path property.

In Figure~\ref{fig:gt_binary} the binarization of $T_g$ in Figure~\ref{fig:gt} according to Definition~\ref{def:binary_set_of_paths}.

\begin{figure}[h]
\captionsetup[subfigure]{justification=centering}
\centering
%FIGURA 1
	\begin{subfigure}{3cm}
\begin{overpic}[width=6cm,percent]{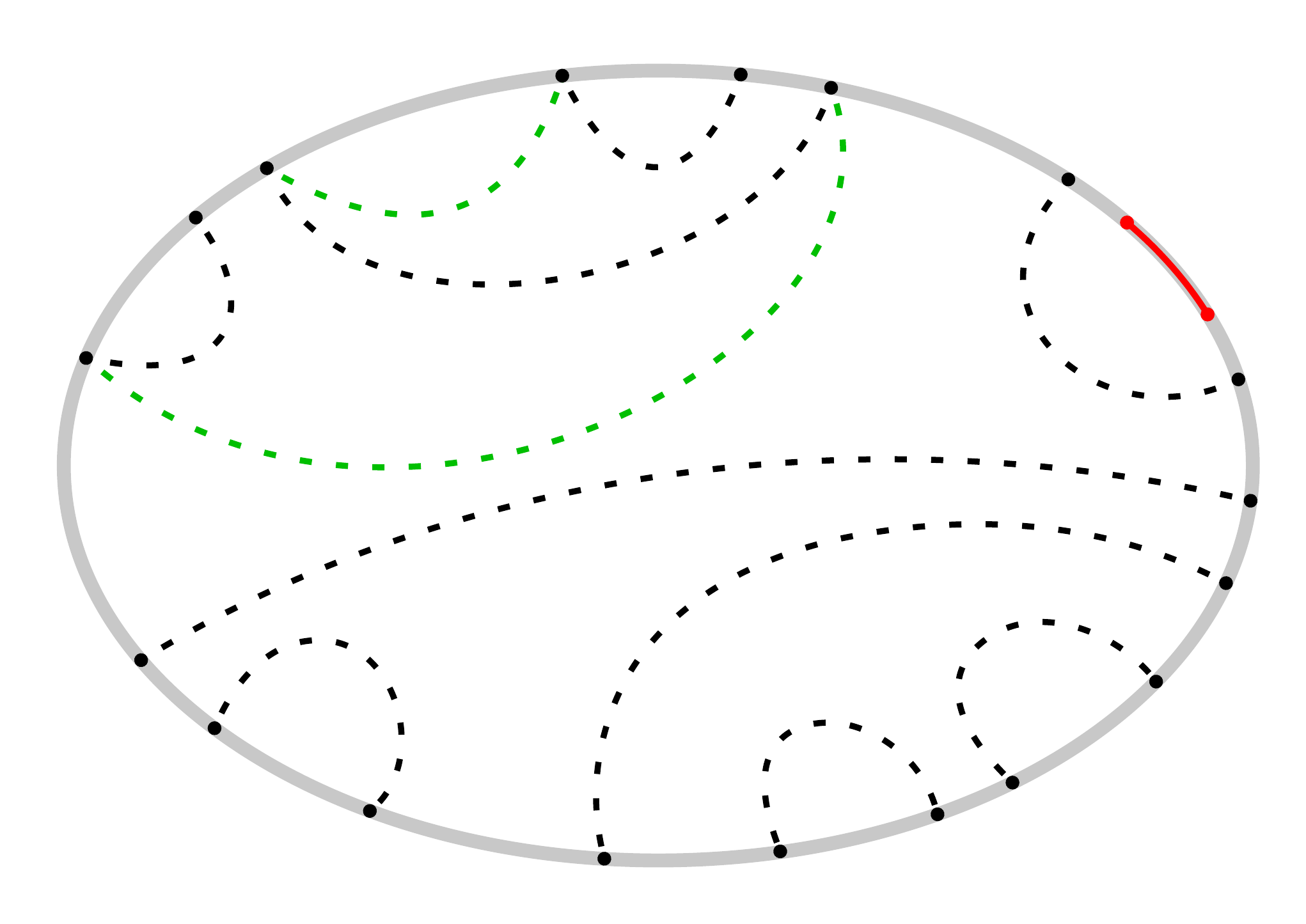}
\put(91,51){\textcolor{myred}{$e_1$}}
\put(-4.5,32){$f^\infty$}

\put(-18.5,43.5){$\textcolor{mygreen}{x_{10}}$ $=x_7$}

\put(64,66.5){$y_8=$ $\textcolor{mygreen}{y_{10}}$}

\put(82.5,59){$y_1$}
\put(96,42){$x_1$}
\put(97.5,30.5){$x_2$}
\put(5.5,16.8){$y_2$}
\put(95.4,24){$x_3$}
\put(44.5,0.3){$y_3$}
%\put(93.5,20.5){$x_4$}
%\put(52.5,0.5){$y_4$}
\put(89.9,16.5){$x_4$}
\put(79,7){$y_4$}
\put(70.2,3.7){$x_5$}
\put(60,0.9){$y_5$}
\put(25,4){$x_6$}
\put(14,9.5){$y_6$}
%\put(-2,44){$x_7$}
\put(8,55){$y_7$}
\put(-3,60.5){$\textcolor{mygreen}{x_{11}}=x_8$}
%\put(64,66.1){$y_8$}
\put(20.5,67.8){$\textcolor{mygreen}{y_{11}}=x_9$}
\put(53,68.3){$y_9$}
\end{overpic}
%\caption{}\label{LABEL1}
\end{subfigure}
\qquad\qquad\qquad\qquad\qquad
%FIGURA 2
	\begin{subfigure}{5cm}
\begin{overpic}[width=6cm,percent]{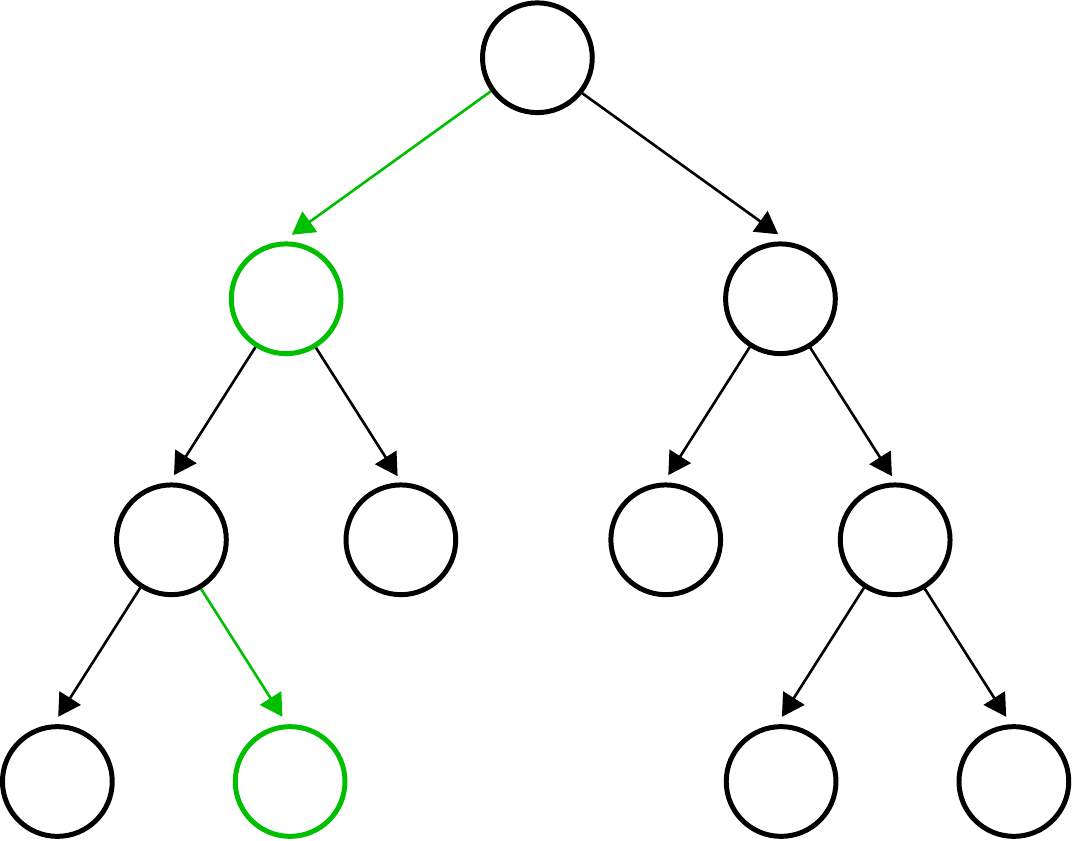}
\put(47.7,72){$p_1$}
\put(23.2,49.5){$p_{10}$}
\put(70.5,49.5){$p_2$}
\put(13.7,27){$p_8$}
\put(35,27){$p_7$}
\put(59.5,27){$p_6$}
\put(81.2,27){$p_3$}
\put(70.5,4.5){$p_5$}
\put(92.2,4.5){$p_4$}
\put(3.,4.5){$p_9$}
\put(23.5,4.5){$p_{11}$}
\end{overpic}
%\caption{}\label{LABEL2}
\end{subfigure}
\caption{how to binarize the genealogy tree in Figure~\ref{fig:gt} according to Definition~\ref{def:binary_set_of_paths} by adding two terminal pairs (in green).}
\label{fig:gt_binary}
\end{figure}
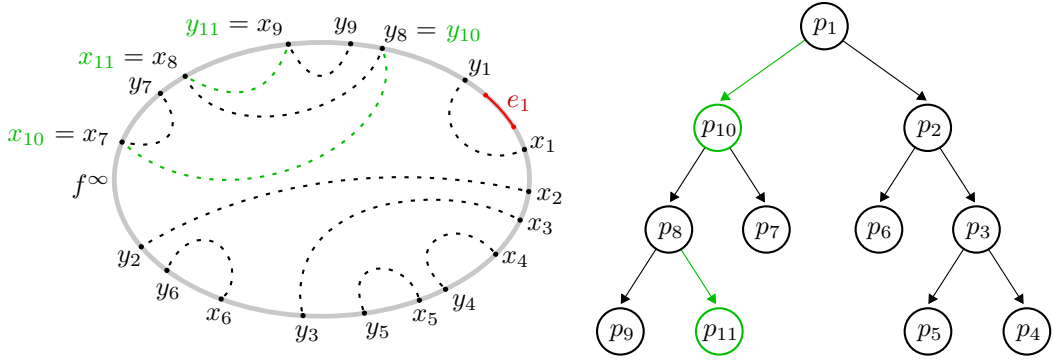

\subsection{Solving faces}\label{sec:solving_faces}

The goal of this subsection is to prove that, given a path labeling $\mathcal{L}$, if every face of $U$ is \emph{solved by $\mathcal{L}$} (see Definition~\ref{def:solved_by_C}) then $\mathcal{L}$ is a forest labeling, as stated in Theorem~\ref{th:faces}. This result allows us to greatly simplify the discussion.

From now on, we assume that our input is a \BNCS $P$ and we denote by $\mathcal{F}$ the set of faces of $U$, where we recall that $U=\bigcup_{p\in P}p$. In the following definition we specify some paths, subpaths and edges related to a face.

\begin{definition}\label{def:upper_path_e_simili}
Given a face $f\in\mathcal{F}$ we define:
\begin{itemize}\itemsep0em
\item $q^f$ the \emph{upper path of $f$} as the minimum path with respect to $\preceq$ satisfying $f\subseteq Int_{q^f}$,
\item $\partial f\setminus q^f$ the \emph{lower boundary of $f$},
\item $e^f_r,e^f_\ell$ the extremal edges of the lower boundary of $f$ such that $e^f_r\in q^f_r$ and $e^f_\ell\in q^f_\ell$.
\end{itemize}  
\end{definition}

The next two lemmas will be used in Theorem~\ref{th:faces}'s proof.

\begin{lemma}\label{lemma:p_cap_f_is_a_path}
Let $p\in P$ and let $f$ be a face in $\mathcal{F}$. Then the intersection between $p$ and $f$ is a path.
\end{lemma}
\begin{proof}
If the intersection between $p$ and $f$ is empty or it consists in a single vertex, then the thesis holds. Thus we assume that there exist two vertices $a,b\in V(p\cap \partial f)$. We have to prove that $p[a,b]\subseteq \partial f$. %, where for a path $q\in P$ and two vertices $x,y$ of $q$, we denote by $q[x,y]$ the subpath of $q$ with extremal vertices $x$ and $y$.

Let $\lambda$ be the $ab$ path on $\partial f$ so that the region $R$ bounded by $\lambda\circ p[a,b]$ does not contain $f$. We have to prove that $\lambda=p[a,b]$. Let us assume by contradiction that there exists $e\in\lambda$ satisfying $e\not\in p$, and let $q\in P$ containing $e$. Being $R$ a closed region, then the non-crossing property implies $a,b\in q$. Thus the single-touch property assures $q[a,b]=p[a,b]$ and hence $e\not\in q$, absurdum.
\end{proof}

\begin{lemma}\label{lemma:binarization}
Let $f$ be a face in $\mathcal{F}$ and let $q$ be the upper path of $f$. Then $q$ intersects $\partial f$ in at least one edge.
\end{lemma}
\begin{proof}
Let us assume by contradiction that $q$ does not intersect $\partial f$ in at least one edge. %Thus or $V(q)\cap V(f)=\emptyset$ or $V(q)\cap V(f)=\{v\}$ for some vertex $v$. 
Let $R=\{r_1,\ldots,r_h\}$ be the minimum set of paths such that $\partial f\subseteq \bigcup_{r\in R}E(r)$, clearly $q\not\in R$. By definition of upper path, $f\subseteq Ext_r$ for all $r\in Q$. Lemma \ref{lemma:p_cap_f_is_a_path} and minimality of $R$ imply that $(r\cap \partial f)\cap (r'\cap \partial f)=\emptyset$, for all $r\neq r'$. Thus $r\npreceq r'$ for all $r\neq r'\in R$, otherwise at least one $r\in R$ should satisfy $f\subseteq Int_r$.

Being $q$ the upper path of $f$, it holds that $q$ is the lowest common ancestor of $r$ and $r'$ for all $r\neq r'\in R$, i.e., all $r\in R$ are children of $q$. Finally, the single-touch property implies that $h\geq3$, otherwise two paths would form the cycle $f$, thus $q$ has at least 3 children, absurdum.
\end{proof}

\begin{definition}\label{def:solved_by_C}
Given a face $f$ in $\mathcal{F}$ and a path labeling $\mathcal{L}$ of $P$, we say that $f$ is \emph{solved by $\mathcal{L}$} if $\mathcal{L}(e^f_r)\cap \mathcal{L}(e^f_\ell)=\emptyset$.
\end{definition}

\begin{theorem}\label{th:faces}
Let $P$ be a \BNCS, and let $\mathcal{L}$ be a path labeling of $P$. If every face of $\mathcal{F}$ is solved by $\mathcal{L}$, then $\mathcal{L}$ is a forest labeling.
\end{theorem}
\begin{proof}
Let us assume by contradiction that every face of $\mathcal{F}$ is solved by $\mathcal{L}$ and that there exists a simple cycle $H$ so that $\bigcap_{e\in E(H)}\mathcal{L}(e)\neq\emptyset$. Let $R_H$ be the region bounded by $H$. If $H$ is a face, then we have a trivial absurdum because $H$ would be solved by $\mathcal{L}$. Otherwise, let $g$ be a face in $R_H$, and let $q$ be the upper path of $g$. By Lemma \ref{lemma:binarization} and being $H$ a cycle, it holds that $q$ intersects $V(H)$. Let $c$ be the maximal subpath of $H$ that is in $Int_q$.

To finish the proof it suffices to show that there exists a face $h$ such that its lower boundary is a subpath of $c$. Indeed, if so, then $\mathcal{L}(e^h_r)\neq \mathcal{L}(e^h_\ell)$ because $h$ is solved by $\mathcal{L}$, therefore $\bigcap_{e\in E(H)}\mathcal{L}(e)=\emptyset$, absurdum.

Let $\mathcal{D}=\{f\in\mathcal{F}\,|\, f\subseteq Int_q\cap R_H$ and $\partial f\cap c\neq\emptyset\}$ and for each face $d\in\mathcal{D}$ let $\Delta_d=\{f\in\mathcal{D}\,|\, f\subseteq Int_{q^d}\}$, we recall that $q^d$ is the upper path of $d$. We have to find a face $h\in\mathcal{D}$ satisfying $|\Delta_h|=1$, i.e., $\Delta_h=\{h\}$; indeed, this implies that the lower boundary of $h$ is a subpath of $c$. 

Now, let $d\in\mathcal{D}$. If $|\Delta_d|=1$, then we have finished, otherwise we observe that for each $d'\in\Delta_d$, it holds that $\Delta_{d'}<\Delta_d$ because $d\not\in\Delta_{d'}$ and $\Delta_d>0$ because $d\in\Delta_d$. Hence there exists a face $h\in\mathcal{D}$ satisfying $|\Delta_h|=1$.
\end{proof}

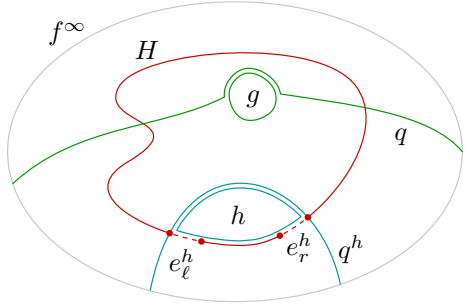
\begin{figure}[h]
\captionsetup[subfigure]{justification=centering}
\centering
%FIGURA 1
	\begin{subfigure}{6cm}
\begin{overpic}[width=6cm,percent]{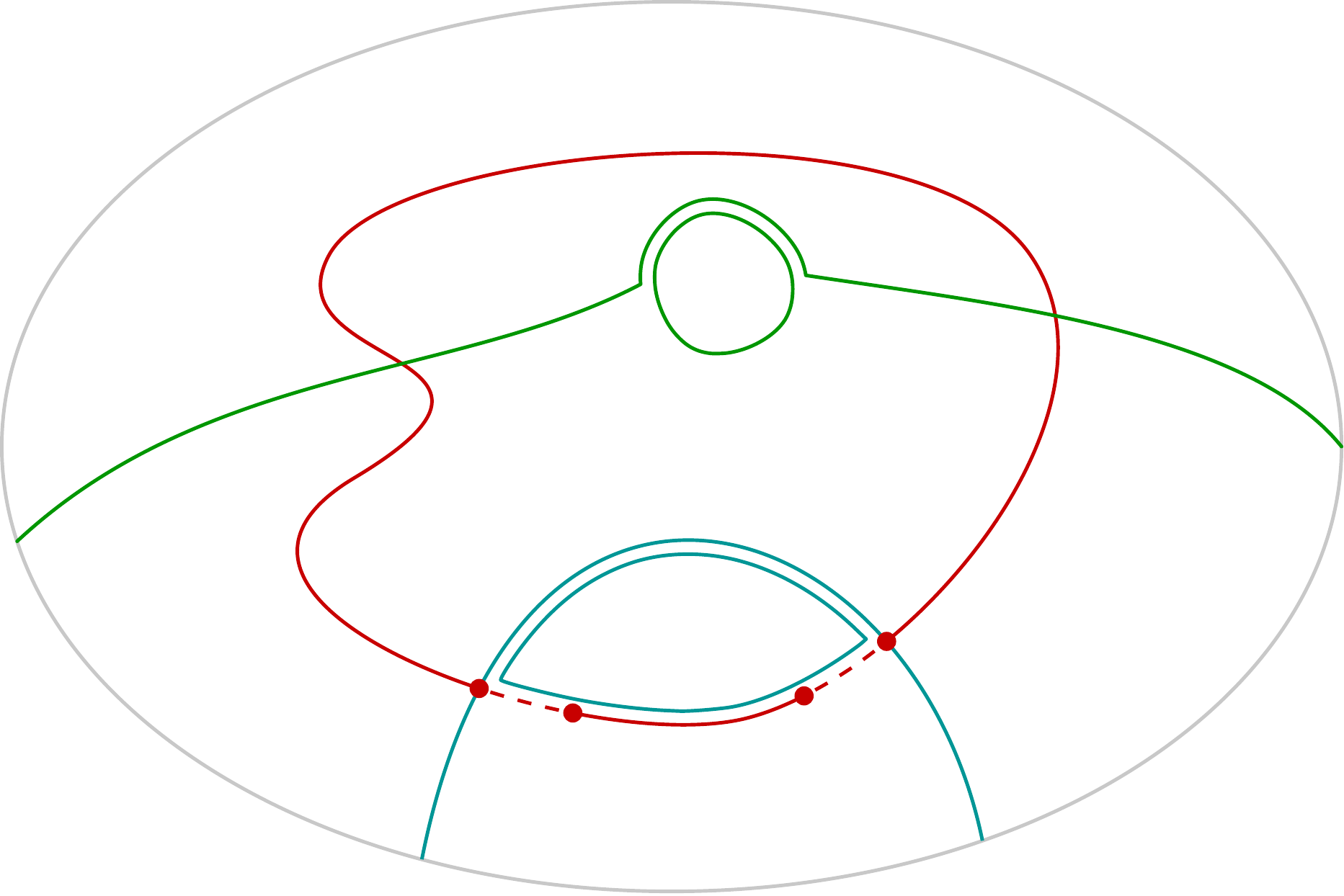}
\put(52.5,44.5){$g$}
\put(49,17.5){$h$}

\put(9,58){$f^\infty$}
\put(28,54){$H$}
\put(85,36){$q$}
\put(72.6,10){$q^h$}

\put(35.5,7){$e^h_\ell$}
\put(61.2,10.5){$e^h_r$}

\end{overpic}
%\caption{Esempi di facce di I, II e III tipo}\label{LABEL1}
\end{subfigure}
  \caption{faces $g,h$, paths $q,q^h$ and cycle $H$ used in the proof of Theorem \ref{th:faces}. Edges $e^h_r$ and $e^h_\ell$ are dotted.}
\label{fig:monochromatich_cycle}
\end{figure}

\section{A first easy upper bound of the Path Covering with Forests Number}\label{sec:15_ALL}

In this section we show that the \FCN of a \NCS $P$ is at most 15. In particular, we present algorithm \FTF that produces a forest labeling of $P$ which uses at most 15 labels. We stress that every result of this section will be used in Section~\ref{sec:4_ALL} to build algorithm \FF which gives an improved upper bound.

In Subsection~\ref{sec:15_outline} we describe the outline of algorithm $\FTF$, in particular we explain in which order we visit paths in $P$ by introducing the sets $\MAX$ and $\TOUCH$. In Subsection~\ref{sec:face_types} we classify the faces related to a path in $\MAX$ in three types. In Subsection~\ref{sec:15_type_I} we deal with faces of the first type and Subsection~\ref{sec:15_type_II_III} with faces of second and third type. Finally, in Subsection~\ref{sec:15_FTF} we exhibit algorithm $\FTF$ and we prove its correctness.

\subsection{Outline of the algorithm}\label{sec:15_outline}

From now on, unless otherwise stated, $P$ denotes a \BNCS. The algorithm \FTF finds a forest labeling of $P$ which uses at most 15 labels. %It is an iterative algorithm but it does not treat all paths in the same way. The goal of this subsection is to explain which paths are labeled at each iteration and on which paths the iterations are made.

\begin{definition}\label{def:max_and_touch}
Given $p\in P$, we define $\Touch_p=\{q\in P\,|$ $q\preceq p$ $\wedge$ $q$ and $p$ share at least one vertex$\}$ and $\Max_p=\{q\in P \,|$ $q\preceq p$ $\wedge$ $q$ is a maximal path w.r.t. $\preceq$ in $ P\setminus \Touch_p\}$. 
\end{definition}

Roughly speaking, $\Touch_p$ consists in all paths in $Int_p$ that ``touch'' $p$, and $\Max_p$ consists in all paths in $Int_p$ that are maximal w.r.t.~$\preceq$ after the removal of paths in $\Touch_p$.

By applying recursively Definition~\ref{def:max_and_touch}, we can separate paths in $P$ in levels. We start by setting $\MAX_0=\{p_1\}$, and then we define sets $\MAX_i$ and $\TOUCH_i$ recursively as follows:

$$
\begin{cases} 
%p_1\in\MAX_0, &\\
p\in\TOUCH_i, &\text{if and only if $p\in\Touch_m$ for some $m\in\MAX_{i-1}$},\\
m\in\MAX_i, &\text{if and only if $m\in\Max_p$ for some $p\in\TOUCH_{i-1}$}.
\end{cases} 
$$

\begin{figure}[h]
\captionsetup[subfigure]{justification=centering}
\centering
%FIGURA 1
	\begin{subfigure}{14cm}
\begin{overpic}[width=14cm,percent]{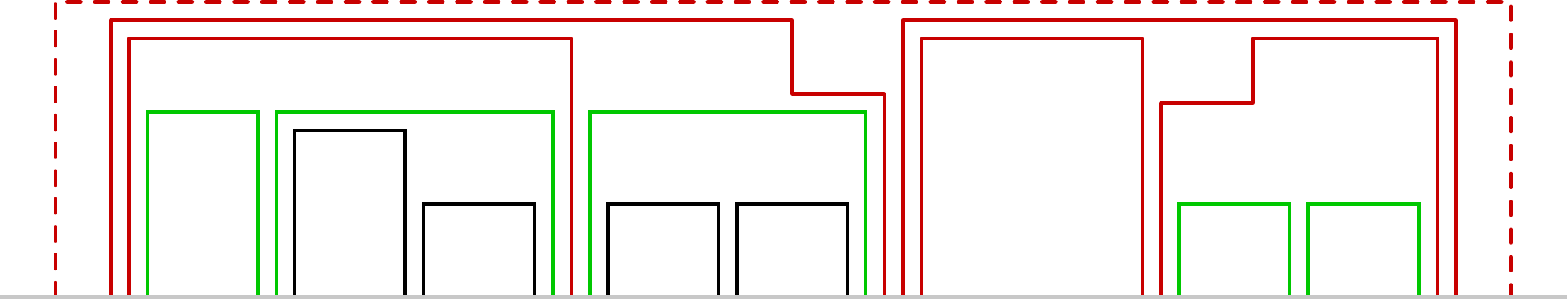}
\put(1.5,18){$p$}
\put(-2,0){$f^\infty$}

\end{overpic}
%\caption{Esempi di facce di I, II e III tipo}\label{LABEL1}
\end{subfigure}
  \caption{$p$ is the dotted path, paths in $\Touch_p$ are red (note that $p\in\Touch_p$), paths in $\Max_p$ are green. Thus if $p\in\MAX_{i-1}$, then red paths are in $\TOUCH_i$, green paths in $\MAX_i$ and black paths in $\TOUCH_{i+1}$.}
\label{fig:dismantle}
\end{figure}

We define also $\MAX=\bigcup_{i\in\mathbb{N}}\MAX_i$ and $\TOUCH=\bigcup_{i\in\mathbb{N}}\TOUCH_i$. For convenience, let $N\in\mathbb{N}$ satisfy $\MAX=\bigcup_{i\in {N}}\MAX_i$ and $\TOUCH=\bigcup_{i\in {N}}\TOUCH_i$; in few words $N$ is the last level. Note that $p\in\Touch_p$ for all $p\in P$, thus $\MAX\subseteq\TOUCH$. These sets are explained in Figure~\ref{fig:dismantle}.

\begin{figure}[h]
\begin{algorithm}[H]
\SetAlgorithmName{$\FTF$}{}{}
\renewcommand{\thealgocf}{}
 \caption{}
  \KwIn{a \NCS $P$}
 \KwOut{a forest labeling $\mathcal{L}:P\mapsto [15]$}
{
Transform $P$ from a \NCS to a \BNCS\;
For each path $p\in P$ define the global variable $\mathcal{L}(p)$ initialized to \NULL\;
For each path $p\in\MAX$ define the global variable $triple(p)$ initialized to \NULL\;
$triple(p_1)\leftarrow T_1$\;
\For{$i=0,\ldots,N$}{
\For{{\normalfont\textbf{each}} $p\in\MAX_i$}{
$\CT(p)$\tcp*{assign $\mathcal{L}(q)\in triple(p)$ for all $q\in\Touch_p$}
$\TM(p)$\tcp*{assign $triple(q)$ for all $q\in\Max_p$}
}
}
\Return $\mathcal{L}$\;
}
\end{algorithm}
\end{figure}

Before explaining the details of algorithm \FTF we state an important remark about two global variables assigned to  paths in $P$. We prefer to use global variables in order to avoid having too much entries in our functions.

\begin{remark}
In our algorithms we use global variables.
\begin{itemize}\itemsep0em
\item For each path $p\in\MAX$ we define the global variable $triple(p)$ that assumes as value a non-ordered triple of labels.

\item For each path $p\in P$ we define the global variable $\mathcal{L}(p)$ that assumes as value a unique label. The output of our main algorithms (algorithm \FTF and algorithm \FF) is $\mathcal{L}$. 
\item Both variables are initialized to \NULL and they do not change once assigned. 
\end{itemize}
\end{remark}

We define also five triples of labels: $T_1=\{1,2,3\}$, $T_2=\{4,5,6\}$, $T_3=\{7,8,9\}$, $T_4=\{10,11,12\}$ and $T_5=\{13,14,15\}$, and let $\mathcal{T}=\{T_1,T_2,T_3,T_4,T_5\}$. %During the execution of algorithm $\FTF$, for each $p\in\MAX$ we assign the variable $triple(p)$ in $\mathcal{T}$, and it does not change once assigned. At the beginning of the algorithm, they are initialized to $\NULL$.  The output of algorithm \FTF is a forest labeling $\mathcal{L}:P\to [15]$. Again, we initialize $\mathcal{L}(p)=\NULL$ for each $p\in P$, and it does not change once assigned.

Now we can describe how algorithm \FTF works. Algorithm \FTF has $N$ iterations and it is based on the two functions $\CT(p)$ and $\TM(p)$. The former assigns one label of $triple(p)$ to $\mathcal{L}(q)$ for all $q\in\Touch_p$, the latter assigns one triple in $\mathcal{T}$ to $triple(q)$ for all $q\in\Max_p$. The assignments are set so that at iteration $i$ all faces in $\Touch_p\cup\Max_p$ are solved by $\mathcal{L}$, for all $p\in\MAX_i$.

We note that the partial order $\preceq$ is respected: if $p\preceq q$, then algorithm \FTF labels $p$ after $q$ is labeled. Let's start with a preliminary definition.

%Now we can describe how algorithm \FTF works. Algorithm \FTF is recursive and it has one level of recursion for each path $p$ in \MAX. At iteration $p$ algorithm \FTF labels all paths in $\Touch_p$ with the labels in $triple(p)$ by using algorithm \CT, and it assigns $triple(m)$ for each $m\in\Max_p$. Then algorithm \FTF calls itself for each $m\in\Max_p$. At iteration $p$ it solves all faces contained in the graph $\bigcup_{q\in\Touch_p\cup\Max_p}q$. Being every face contained in $\bigcup_{q\in\Touch_p\cup\Max_p}q$, for some $p\in\MAX$, then algorithm \FTF solves every face. For the initialization, algorithm \FTF starts at iteration $p_1$, where $triple(p_1)$ is chosen arbitrary in $\mathcal{T}$.

\subsection{Face types}\label{sec:face_types}

Given $p\in\MAX$, we classify all faces whose upper path is in $\Touch_p$ in three types as shown in Definition~\ref{def:face_types} and in Figure~\ref{fig:face_types}. We observe that, given a face $f$ of $\Touch_p\cup\Max_p$, the upper path $q$ of $f$ is in $\Touch_p$ because the children of each path in $\Max_p$ are not in $\Touch_p\cup\Max_p$. This fact is crucial to classify the faces in the following definition.

\begin{definition}\label{def:face_types}
Let $p\in\MAX$. Let $f$ be a face of $\Touch_p\cup\Max_p$ and let $q$ be the upper path of $f$. We say that
\begin{itemize}\itemsep0em
\item $f$ is \emph{of type I for $p$} if $q_r,q_\ell\in \Touch_p$,
\item $f$ is \emph{of type II for $p$} if $q_r,q_\ell\in \Max_p$. For convenience, we denote $q_r$ and $q_\ell$ by $m^f_r$ and $m^f_\ell$, respectively,
\item $f$ is \emph{of type III for $p$} if $q_r\in \Touch_p$ and $q_\ell\in \Max_p$ (or $q_r\in \Max_p$ and $q_\ell\in \Touch_p$).
\end{itemize}
If no confusion arises, then we omit the reference to $p$.
\end{definition}

\begin{figure}[h]
\captionsetup[subfigure]{justification=centering}
\centering
%FIGURA 1
	\begin{subfigure}{4.1cm}
\begin{overpic}[width=4.1cm,percent]{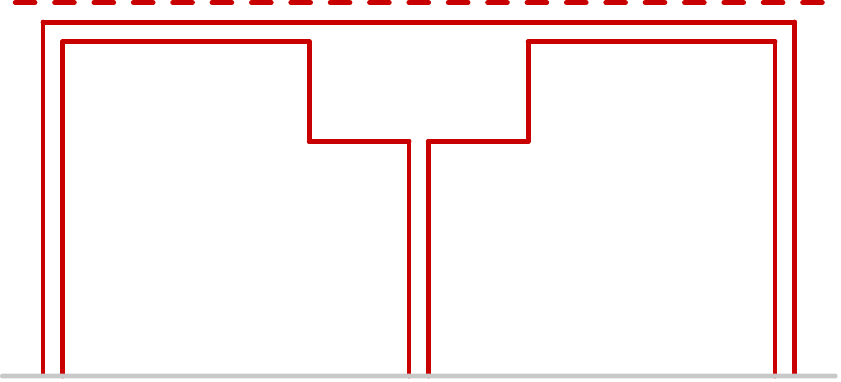}

\put(-5,42.5){$p$}
\put(-9,0){$f^\infty$}
\put(46,33){$f_{I}$}

\end{overpic}
%\caption{}\label{fig:face_types_I}
\end{subfigure}
\qquad
%FIGURA 2
	\begin{subfigure}{4.1cm}
\begin{overpic}[width=4.1cm,percent]{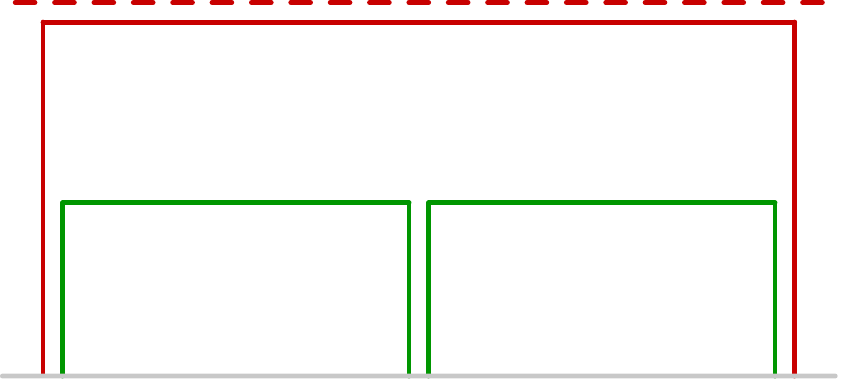}

\put(-5,42.5){$p$}
\put(-9,0){$f^\infty$}
\put(46,33){$f_{II}$}

\end{overpic}
%\caption{}\label{fig:face_types_II}
\end{subfigure}
\qquad
%FIGURA 3
	\begin{subfigure}{4.1cm}
\begin{overpic}[width=4.1cm,percent]{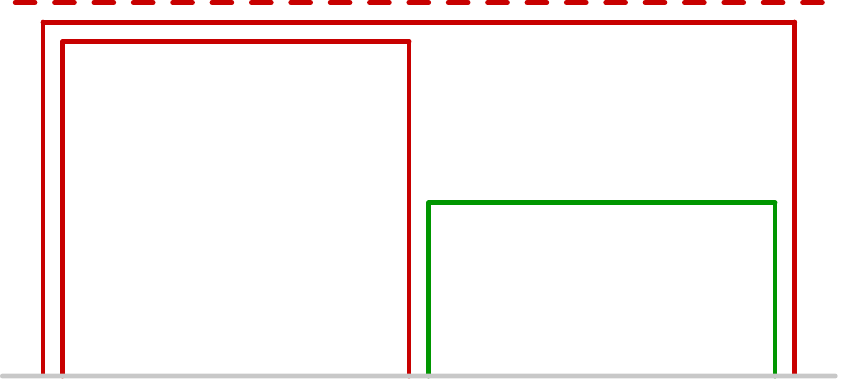}

\put(-5,42.5){$p$}
\put(-9,0){$f^\infty$}
\put(66,30){$f_{III}$}

\end{overpic}
%\caption{}\label{fig:face_types_III}
\end{subfigure}
  \caption{$f_I$, $f_{II}$, $f_{III}$ are faces of type I, II, III, respectively, for $p$.}
\label{fig:face_types}
\end{figure}

\subsection{Dealing with faces of type I }\label{sec:15_type_I}

The main result of this subsection is given in Proposition~\ref{prop:Color}, that allows us to assign $\mathcal{L}(q)$ for all paths $q$ in $\Touch_p$ with the three labels in $triple(p)$ so that every face of type I for $p$ is solved by $\mathcal{L}$, for an arbitrary $p\in\MAX$. We obtain such a labeling by using the two algorithms \RC and \CT. Let's start with a definition, explained in Figure~\ref{fig:R(q)_L(q)}, about paths in $\Touch_p$.

\begin{definition}\label{def:R(q)_L(q)}
Let $p\in\MAX$. We define $\text{Tree}\Touch_p$ as the rooted subtree of $T_g$ induced by $\Touch_p$. Given $q\in\Touch_p$ we define $R^p(q)$ as the set of \emph{right descendants of $q$ w.r.t. $p$} and it is composed by all elements in $\text{Tree}\Touch_p$ that are in the path from $q$ to the rightmost leaf of the subtree of $\text{Tree}\Touch_p$ rooted at $q$. Similarly, we define $L^p(q)$ as the set of \emph{left descendants of $q$ w.r.t. $p$} and it is composed by all elements in $\text{Tree}\Touch_p$ that are in the path from $q$ to the leftmost leaf of the subtree of $\text{Tree}\Touch_p$ rooted at $q$. When no confusion arises, we denote $R^p(q)$ and $L^p(q)$ by $R(q)$ and $L(q)$, respectively.
\end{definition}

\begin{figure}[h]
\captionsetup[subfigure]{justification=centering}
\centering
%FIGURA 1
	\begin{subfigure}{5.5cm}
\begin{overpic}[width=5.5cm,percent]{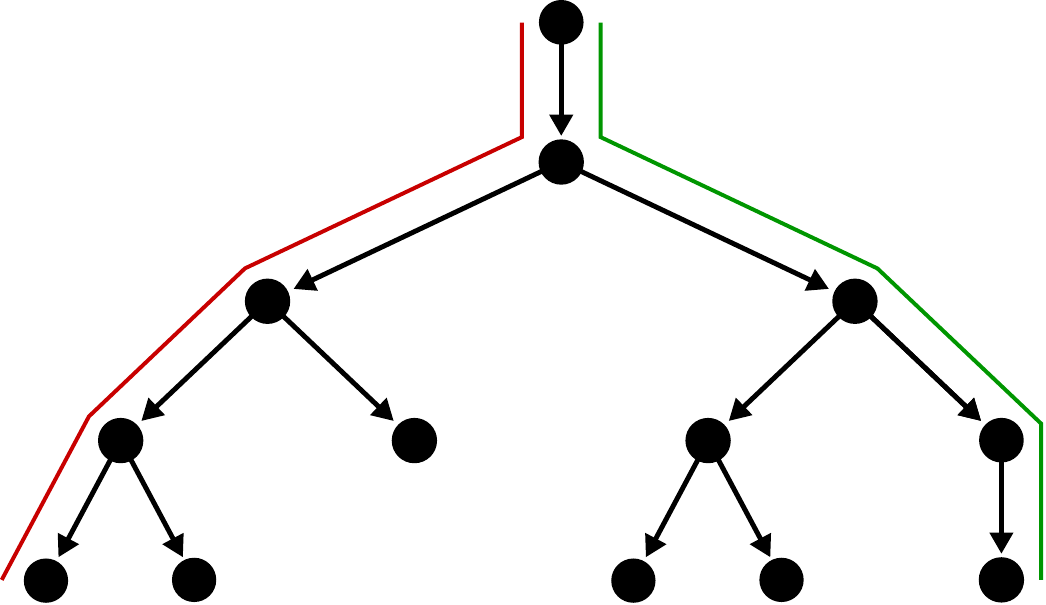}
\put(23,40){$L(q)$}
\put(70,40){$R(q)$}
\put(52.5,60){$q$}
\end{overpic}
%\caption{}\label{LABEL1}
\end{subfigure}
  \caption{the subtree of $TreeTouch_p$ rooted at $q$. In red vertices belonging to $L(q)$ and in green those to $R(q)$. Note that $q$ and its child are both in $L(q)\cap R(q)$.}
\label{fig:R(q)_L(q)}
\end{figure}

In order to solve faces of type I for $p$ we deal with intersecting paths. Given two paths $p$ and $q$ with $q\preceq p$, roughly speaking, it holds that if $q$ intersects $p$ on vertices, then $q$ splits $p$ into three subpaths: the right subpath containing $x_p$, the subpath $p\cap q$ (this is a path by the single-touch property), and the left subpath containing $y_p$. Note that the right subpath or the left subpath might be composed by one vertex. From this fact we obtain the following  remark, that can be formally proved by using Jordan's Curve Theorem~\cite{jordan} and by observing that $p\circ\gamma_p$ is a closed curve for every arbitrary path $p$ in $P$.

\begin{remark}\label{remark:crucial}
Let $p,q,q'\in P$ satisfy $q\preceq p$, $q'\preceq p$ and $p\cap q\neq\emptyset$. Then $q$ splits $p$ into three subpaths: $p=r\circ p\cap q\circ \ell$, with $x_p\in r$ and $y_p\in\ell$. 
%where $r$ and $\ell$ and the \emph{right subpath} and \emph{left subpath}, respectively. 
The following statements hold:
\begin{itemize}\itemsep0em
\item if $q'\cap r\neq\emptyset$, then $q\lhd q'$,
\item if $q'\cap \ell\neq\emptyset$, then $q'\lhd q$.
\end{itemize}
\end{remark}

We want to apply the previous result to faces of type I for $p$. Given a face $f$ of type I for $p$, we observe that both $e^f_r$ and $e^f_\ell$ have an extremal vertices on $p$ by their definition. Thus if a path $q$ contains either $e^f_r$ or $e^f_\ell$, then $q\in\Touch_p$. The following lemma, whose proof is strictly based on Remark~\ref{remark:crucial}, explains which paths contain $e^f_r$ and which paths contains $e^f_r$. This result is the key to label all paths in $\Touch_p$ with three labels in order to solve all face of type I for $p$ (see Lemma~\ref{lemma:RC} and Proposition~\ref{prop:Color}).

For an edge $e$ we define $P(e)=\{p\in P\ |\ e\in p\}$.

\begin{lemma}\label{lemma:right_left}
Let $p\in\MAX$. Let $f$ be a face of type I for $p$ and let $q$ be the upper path of $f$, then $P(e^f_r)\subseteq L(q_r)$ and $P(e^f_\ell)\subseteq R(q_\ell)$.
\end{lemma}
\begin{proof}
We prove that $P(e^f_r)\subseteq L(q_r)$, by symmetry, it also proves that $P(e^f_\ell)\subseteq R(q_\ell)$. First of all, $e^f_r\in q_r$ by Definition~\ref{def:upper_path_e_simili}. Let $z$ be the extremal vertex of $e^f_r$ not belonging to $p$. Then the subpath $\lambda$ of $q_r$ from $z$ to $y_{q_r}$ does not intersect $p$ on vertices because of the single-touch property. 

Let us assume that $q_r$ has two children $c_r$ and $c_\ell$ belonging to $\Touch_p$, with $c_\ell\lhd c_r$. Indeed, if $q_r$ has no children belonging to $\Touch_p$, then the thesis is trivial, and if $q_r$ has exactly one child belonging to $\Touch_p$, then it belongs to $L(q_r)$ by definition.

Being $\lambda\cap p=\emptyset$, then $e^f_r$ belongs to $c_\ell$; indeed, if $e^f_r$ belongs to $c_r$, then Remark~\ref{remark:crucial} would imply $c_\ell\cap q_r=\lambda$, thus $c_\ell\not\in\Touch_p$, absurdum. By repeating recursively this reasoning, the thesis follows.
\end{proof}

Now we introduce recursive algorithm \RC whose entries are $p,q,\sigma$, where $p$ is a path in $\MAX$, $q$ is a path in $\Touch_p$ and $\sigma$ is an permutation of $triple(p)$. By calling $\RC(p,q,\sigma)$, we label all paths in $\Touch_p$ belonging to $Int_q$. It holds that $q$ is labeled with the first label in $\sigma$. Then if $q$ has one child in $\Touch_p$, then $\sigma$ is passed to its child without permutations. If $q$ has two children in $\Touch_p$, then $\sigma$ is passed to its children with permutations. An example of how these permutations change is given in Figure~\ref{fig:rightcolor_and_leftcolor}.

\begin{figure}[h]
\begin{algorithm}[H]
\SetAlgorithmName{$\RC(p,q,(c_1,c_2,c_3))$}{}{}
\renewcommand{\thealgocf}{}
\caption{}
 \KwIn{a path $p$ in $\MAX$, a path $q$ in $Touch_p$ and a permutation $( c_1,c_2,c_3)$ of $triple(p)$}
  \KwOut{assign $\mathcal{L}(q')\in triple(p)$ for all $q'\in\Touch_p$ such that $q'\preceq q$}
{$\mathcal{L}(q)\leftarrow c_1$\label{line:assign_color}\;
\If{$q$ has one child $q'$ in $\Touch_p$}{$\RC(p,q',( c_1,c_2,c_3) )$\;}\label{line:one_child}
\If{$q$ has two children in $\Touch_p$}{
$\RC(p,q_r,( c_1,c_3,c_2))$\;\label{line:call_Color_right}
$\RC(p,q_\ell,( c_2,c_1,c_3))$\;\label{line:call_Color_left}
}
%\Return $\mathcal{L}$\;
}
\end{algorithm}
\end{figure}

\begin{figure}[h]
\captionsetup[subfigure]{justification=centering}
\centering
\begin{subfigure}{7cm}
\begin{overpic}[width=7cm,percent]{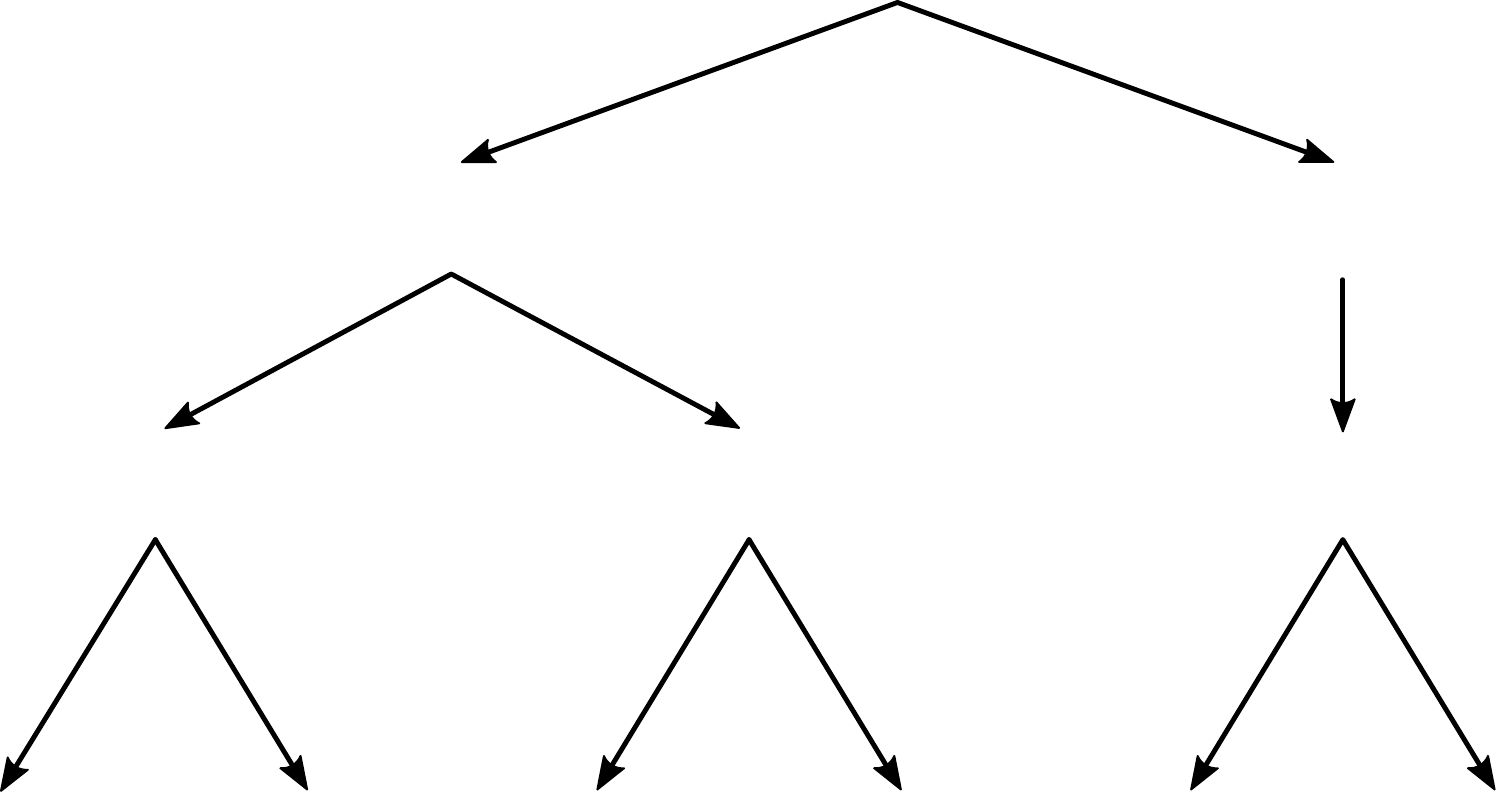}
\put(51.3,55.5){$( {\color{red} 1},2,3)$}

\put(21.6,37){$( {\color{red} 2},1,3)$}
\put(81.1,37){$( {\color{red} 1},3,2)$}

\put(1.8,19){$( {\color{red} 1},2,3)$}
\put(41.4,19){$( {\color{red} 2},3,1)$}
\put(81.1,19){$( {\color{red} 1},3,2)$}

\put(-8.6,-5){$( {\color{red} 2},1,3)$}
\put(11.8,-5){$( {\color{red} 1},3,2)$}
\put(31,-5){$( {\color{red} 3},2,1)$}
\put(51.5,-5){$( {\color{red} 2},1,3)$}
\put(70.9,-5){$( {\color{red} 3},1,2)$}
\put(91.4,-5){$( {\color{red} 1},2,3)$}

\put(-6,55.5){\RC}
\end{overpic}
\end{subfigure}
\vspace{3mm}
\caption{assume that $triple(p)=\{1,2,3\}$. Example of permutations of $triple(p)$ during the call $\RC(p,q,( 1,2,3))$, where $q$ is the vertex on the top. The first element of the permutation, in red, denotes also the label assigned to the path.}
\label{fig:rightcolor_and_leftcolor}
\end{figure}

The recursive calls of \RC change the permutation of labels in order to obtain the following lemma. In Subsection~\ref{sec:4_type_I} we will introduce its symmetric version called algorithm \LC.

\begin{lemma}\label{lemma:RC}
Let $p\in\MAX$, let $( c_1,c_2,c_3)$ be any ordering of $triple(p)$ and let $q\in\Touch_p$. If we call $\RC(p,q,( c_1,c_2,c_3))$, then $\mathcal{L}(R(q))=\{c_1\}$ and $\mathcal{L}(L(q))\subseteq\{c_1,c_2\}$.
\end{lemma}
\begin{proof}
To prove that $\mathcal{L}(R(q))=\{c_1\}$, we observe that the first element of the permutation does not change for all the calls in $R(q)$ because of Line \ref{line:one_child} and Line \ref{line:call_Color_right}. %Thus $\mathcal{L}(R(q))=c_1$ by Line \ref{line:assign_color}. 

Similarly, to prove that $\mathcal{L}(L(q))\subseteq\{c_1,c_2\}$, we observe that the first and the second element of the permutation are the same (possibly swapped) for all the calls in $L(q)$ because of Line \ref{line:one_child} and Line \ref{line:call_Color_left}.
\end{proof}

By using algorithm \RC, we introduce the following compact algorithm.

\begin{figure}[h]
\begin{algorithm}[H]
\SetAlgorithmName{$\CT(p)$}{}{}
\renewcommand{\thealgocf}{}
\caption{}
 \KwIn{a path $p$ in $\MAX$}
 \KwOut{assign $\mathcal{L}(q)\in triple(p)$ for all $q$ in $\Touch_p$ so that every face of type I for $p$ is solved by $\mathcal{L}$}
{Let $(c_1,c_2,c_3)$ be an arbitrary ordering of $triple(p)$\;
$\RC(p,p,(c_1,c_2,c_3))$\tcp*{assign $\mathcal{L}(q)\in triple(p)$ for all $q\in\Touch_p$}
%\Return $\mathcal{L}$\;
}
\end{algorithm}
\end{figure}

\begin{prop}\label{prop:Color}
Let $p\in\MAX$, if we call $\CT(p)$, then $\mathcal{L}(q)\in triple(p)$ for all $q\in\Touch_p$ and every face of type I for $p$ is solved by $\mathcal{L}$.
\end{prop}
\begin{proof}
It's clear that if we call $\CT(p)$, then all paths in $\Touch_p$ are labeled with labels in $triple(p)$. Now let $f$ be a face of type I for $p$ and let $q$ be the upper path of $f$. By definition of face of type I, $q\in\Touch_p$. We have to prove that $f$ is solved by $\mathcal{L}$. Note that algorithm \CT calls recursively algorithm \RC, thus by calling $\CT(p)$ we arrive to call $\RC(p,q,( c_1,c_2,c_3))$, for some $( c_1,c_2,c_3)$ permutation of $triple(p)$.

The call $\RC(p,q,( c_1,c_2,c_3))$ implies the call $\RC(p,q_r,( c_1,c_3,c_2) )$ and the call $\RC(p,q_\ell,( c_2,c_1,c_3) )$. Lemma~\ref{lemma:RC} implies $\mathcal{L}(L(q_r))\subseteq\{c_1,c_3\}$ and $\mathcal{L}(R(q_\ell))=\{c_2\}$. Being $\mathcal{L}(e^f_r)\subseteq \mathcal{L}(L(q_r))$ and $\mathcal{L}(e^f_\ell)\subseteq \mathcal{L}(R(q_\ell))$ by Lemma~\ref{lemma:right_left}, then $\mathcal{L}(e^f_r)\cap \mathcal{L}(e^f_\ell)=\emptyset$, hence $f$ is solved by $\mathcal{L}$.
\end{proof}

\subsection{Dealing with faces of type II and type III}\label{sec:15_type_II_III}

In this subsection we build algorithm $\TM(p)$ that assigns $triple(q)$ for all $q\in\Max_p$, where $p\in\MAX$. The consequences of the execution of $\TM(p)$ are explained in Proposition~\ref{prop:TripleMax} and they concern with paths in $\Touch_p$ and $\Max_p$ that \emph{interfere with} (see Definition~\ref{def:interfere}) the same face. %In few words, Proposition~\ref{prop:TripleMax} assures that paths in $\Max_p$ if two paths $m,m'$ in $\Max_p$ interfere with the same face, then $triple(m)\cap triple(m')=\emptyset$. 

To solve the faces of type I for $p$ we worked only with paths in $\Touch_p$, because given $f$ of type I for $p$, then $q$ contains $e^f_r$ or $e^f_\ell$ only if $q\in\Touch_p$. To deal with faces of type II and type III we have to work also with paths in $\Max_p$. For this reason we introduce the concept of \emph{interfere with}.

\begin{definition}\label{def:interfere}
Let $f\in\mathcal{F}$, we say that $q\in P$ \emph{interferes with $f$} if $q$ contains either $e^f_r$ or $e^f_\ell$.
\end{definition}

The structure of faces of type II for $p$ is easy. Indeed, given a face $f$ of type II for $p$, then $m^f_r$ and $m^f_\ell$ interfere with $f$ and it does not exist any $q\in \Max_p\cup\Touch_p$ interfering with $f$ so that $q\not\in\{m^f_r,m^f_\ell\}$. Clearly, some paths in $\Touch_{m^f_r}$ and $\Touch_{m^f_\ell}$ may interfere with $f$, but they are labeled by $\CT(m^f_r)$ and $\CT(m^f_\ell)$ (see algorithm \FTF). Thus to solve the face $f$ it suffices to set $triple(m^f_r)\cap triple(m^f_\ell)=\emptyset$. 

Dealing with faces of type III is more complex. For convenience, we define $\Max_p^{II}=\{m\in\Max_p \,|\, f_m$ is a face of type II for $p\}$ and $\Max_p^{III}=\{m\in\Max_p \,|\, f_m$ is a face of type III for $p\}$.

To explain the following definition we note that, given a path $m\in\MAX\setminus\{p_1\}$, $m$ forms a face $f$ with its parent and its sibling (the other child of the parent). In order to solve $f$, we deal with the extremal edge of the lower boundary of $f$ belonging to $m$. Clearly, this does not happen for $p_1$, for which we choose as arbitrary edge the edge adjacent to $x_1$.

\begin{definition}\label{def:special_edge}
Given $m\in\MAX\setminus\{p_1\}$ we define $f_m\in\mathcal{F}$ the face whose upper path is the parent of $m$ in $T_g$. Moreover, we define $\se^m$ the edge among $e^{f_m}_r$ and $e^{f_m}_\ell$ belonging to $m$. Finally, we define $\se^{p_1}$ as the edge in $p_1$ adjacent on $x_1$.
\end{definition}

Note that, given $m\in\Max_p^{III}$, $\se^m$ belongs only to $m$---and so, possibly, to some paths in $\Touch_m$---and there are not other paths in $\Touch_p\cup\Max_p$ containing $\se^m$. But there may exist a path $m'\in\Max^{II}_p\cup\Max_p^{III}$ distinct from $m$ interfering with $f$. Thus we introduce the relation $\dashrightarrow$ and we study its structure in Lemma~\ref{lemma:forest}. We observe that $\dashrightarrow$ is not transitive.

%Osserviamo che dato $m\in\Max_p^{III}$, allora $\se^m$ appartiene solamente ad $m$, ovvero non esiste $q\in P$ tale che $\se^m\in q$ e $q\not\in\Touch_m$. Quindi per risolvere $f_m$ sara' sufficiente porre $triple(m)\cap \mathcal{L}(q)=\emptyset$ per ogni $q$ che interferes with $f_m$. Useremo altre due triple per gli elementi di $\Max_p^{III}$ diverse da quelle usate per colorare $\Touch_p$ e per colorare $\Touch_m$ per gli $m$ tali che $f_m$ e' di secondo tipo. Quindi per ogni $m\in\Max_p^{III}$ dobbiamo vedere quali elementi di $\Max_p^{III}$ interfere with $f_m$ e studiarne la struttura. Per questo diamo la definizione di $\dashrightarrow$ e nel Lemma~\ref{lemma:forest} dimostriamo che the graph $(\Max_p^{III},\dashrightarrow)$ is a rooted forest.

\begin{definition}\label{def:dasharrow}
Let $p\in\MAX$. Given $m,m'\in\Max_p$ we write $m\dashrightarrow m'$ if $m$ interferes with $f_{m'}$.
\end{definition}

Figure~\ref{fig:dasharrow} explains Definition~\ref{def:dasharrow}, note that $m_3\dashrightarrow m_5$ even if $m_3$ and $m_5$ are vertex disjoint.

\begin{figure}[h]
\captionsetup[subfigure]{justification=centering}
\centering
%FIGURA 1
	\begin{subfigure}{8cm}
\begin{overpic}[width=8cm,percent]{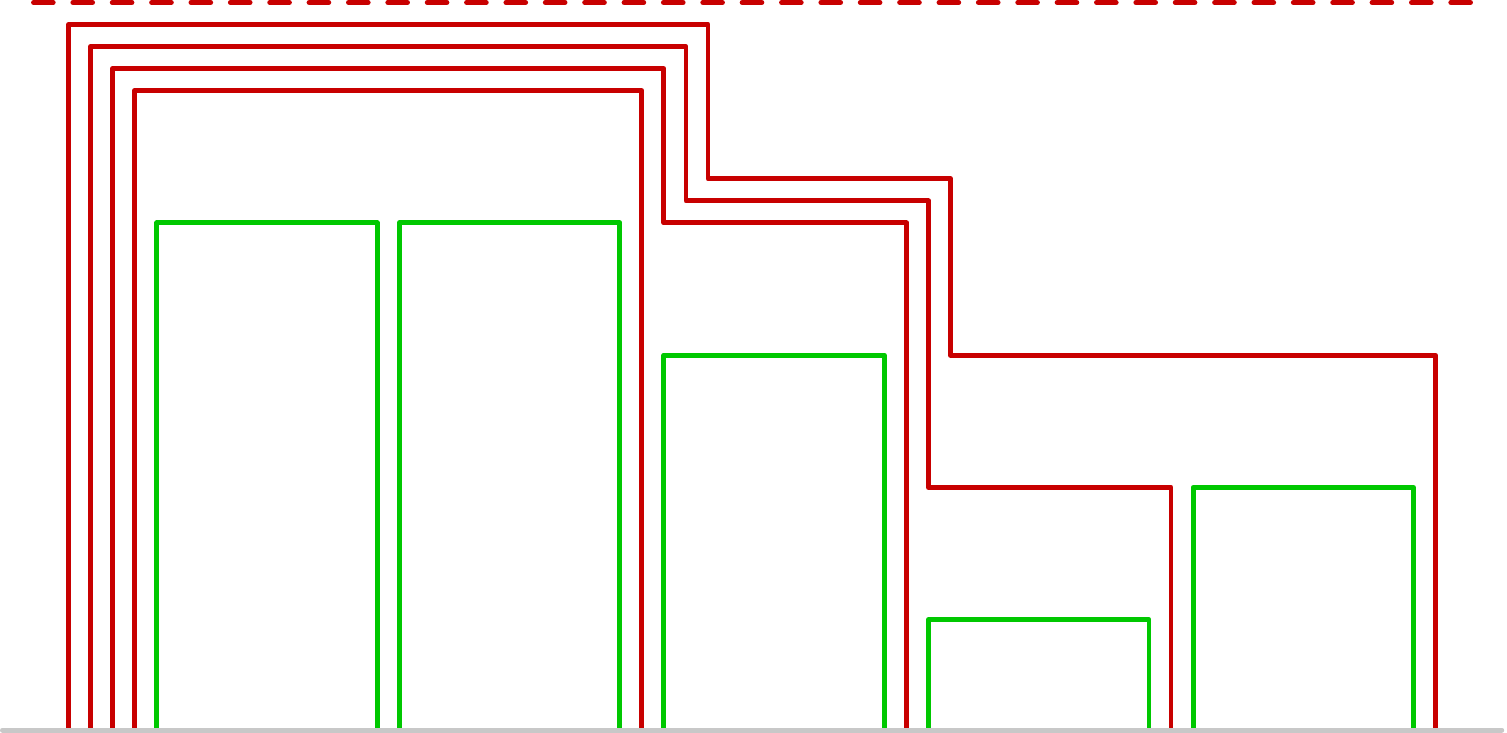}
\put(24,37){$f$}
\put(48,28){$f_{m_3}$}
\put(65,10.5){$f_{m_4}$}
\put(75,19.5){$f_{m_5}$}

\put(-1.3,47){$p$}
\put(-4.5,0){$f^\infty$}

\put(15,30){$m_1$}
\put(31,30){$m_2$}
\put(49,21.5){$m_3$}
\put(66.5,4){$m_4$}
\put(84,12.5){$m_5$}

\end{overpic}
%\caption{}\label{LABEL1}
\end{subfigure}
\qquad\qquad
\begin{subfigure}{2cm}
\begin{overpic}[width=2cm,percent]{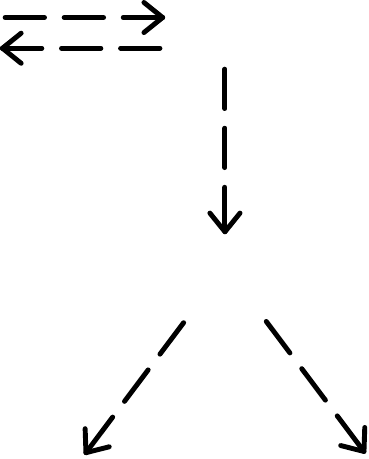}
\put(-23,90){$m_1$}
\put(42,90){$m_2$}
\put(40.5,35){$m_3$}
\put(12,-8){$m_4$}
\put(75,-8){$m_5$}

\end{overpic}
%\caption{}\label{LABEL1}
\end{subfigure}

\caption{on the left in red paths in $\Touch_p$ and in green paths in $\Max_p$. On the right the graph $(\Max_p,\dashrightarrow)$. Note that $(\Max_p^{III},\dashrightarrow)$ is a rooted tree, $f$ is a face of type II for $p$ while $f_{m_3}$ and $f_{m_5}$ are faces of type III for $p$.}
\label{fig:dasharrow}
\end{figure}

The main consequence of the following lemma is that $(\Max_p^{III},\dashrightarrow)$ is bipartite. This fact will be used in the algorithm $\TM$.

\begin{lemma}\label{lemma:forest}
For each $p\in\MAX$, the graph $(\Max_p^{III},\dashrightarrow)$ is a forest of rooted trees.
\end{lemma}
\begin{proof}
For each $m\in\Max_p^{III}$ we denote by $q_m$ its parent in $T_g$. We split the proof into three parts.
%\todob{usato vspace}
\begin{enumerate}[label=\alph{CONT})]\stepcounter{CONT}\itemsep0em
%\begin{enumerate}[label=\alph*),ref=\alph*)]\itemsep0em
\item\label{claim_1} let $m,m'\in\Max_p$. If $m\dashrightarrow m'$, then $q_m\preceq q_{m'}$ and $q_m$ interferes with $f_{m'}$, 
\end{enumerate}
\vspace*{-4.5mm}
\begin{enumerate}[label=\alph{CONT})]\stepcounter{CONT}\itemsep0em
\item\label{claim_2} every vertex in $(\Max_p,\dashrightarrow)$ has at most one incoming dart,
\end{enumerate}
\vspace*{-5mm}
\begin{enumerate}[label=\alph{CONT})]\stepcounter{CONT}\itemsep0em
\item\label{claim_3} in $(\Max_p^{III},\dashrightarrow)$ there are no cycles.
\end{enumerate}
\claimbegin{\ref{claim_1}} 
It holds that $f_{m'}\subseteq Int_{q_{m'}}$, thus the lower boundary of $f_{m'}$ is in $Int_{q_{m'}}\setminus q_{m'}$. This implies that $q_m$ is in $Int_{q_{m'}}\setminus q_{m'}$ because otherwise $m$ could not satisfy $m\dashrightarrow m'$. Therefore, $q_m\preceq q_{m'}$ as we claimed. Moreover, $q_m$ interferes with $f_{m'}$ for the same reasoning.
\claimend{\ref{claim_1}}

\claimbegin{\ref{claim_2}} 
Let us assume by contradiction that there exists $m\in\Max$ having two incoming darts in $(\Max_p,\dashrightarrow)$. Then, by definition of $\dashrightarrow$, there exist $m',m''\in\Max_p$ sharing the extremal edge $e$ of the lower boundary of $f_m$ not contained in $m$. Being $m'\circ \gamma_{m'}$ and $m''\circ \gamma_{m''}$ two closed curves, then every face $f$ containing $e$ is either in $Int_m$ or in $Int_{m'}$. Thus $f$ is not a face of type I, nor II nor III for $p$, absurdum.
\claimend{\ref{claim_2}}

\claimbegin{\ref{claim_3}}
Let us assume by contradiction that there exist $r$ elements of $\Max_p^{III}$, $m_1,m_2,\ldots,m_r$, $r\geq 2$ such that $m_i\dashrightarrow m_{i+1}$, for all $i\in[r-1]$ and $m_r\dashrightarrow m_1$. Then \ref{claim_1} implies $q_{m_1}\preceq q_{m_2}\preceq\ldots\preceq q_{m_r}\preceq q_{m_1}$, and thus $q_{m_1}=q_{m_2}=\ldots=q_{m_r}$. Thus $q_{m_1}$ has at least $r$ children. If $r\geq3$, then it is absurdum because $T_g$ is a binary tree. Else, $r=2$ and thus $f_{m_1}=f_{m_2}$ is a face of type II because $m_1$ and $m_2$ share the same parent in $T_g$, implying $m_1\not\in\Max_p^{III}$ and $m_2\not\in\Max_p^{III}$, absurdum. 
\claimend{\ref{claim_3}}
Now, we observe that the thesis is a consequences of \ref{claim_2} and~\ref{claim_3}.
\end{proof}

\begin{figure}[h]
\begin{algorithm}[H]
\SetAlgorithmName{$\TM(p)$}{}{}
\renewcommand{\thealgocf}{}
 \caption{}
  \KwIn{a path $p$ in $\MAX$}
  \KwOut{assign $triple(q)\in\mathcal{T}$ for all $q\in\Max_p$}
{Let $\{X,Y,W,Z\}=\mathcal{T}\setminus\{triple(p)\}$\label{line:15_definition:X_Y_W_Z}\;
\For{{\normalfont\textbf{each}} face $f$ of type II for $p$\label{line:15_for_II_type}}{
$triple(m^f_r)\leftarrow X$\;
$triple(m^f_\ell)\leftarrow Y$\;
}
Bipartite the forest $(\Max_p^{III},\dashrightarrow)$ into the two classes $A$ and $B$ so that $u\not\dashrightarrow v$ for all $u,v$ in the same class\label{line:15_bipartite}\;
\For{{\normalfont\textbf{each}} $m\in A$}{
$triple(m)\leftarrow W$\label{line:15_A}}
\For{{\normalfont\textbf{each}} $m\in B$}{
$triple(m)\leftarrow Z$\label{line:15_B}}
%\Return $triple(m)$ for all $m\in\Max_p$\; 
}
\end{algorithm}
\end{figure}

The following lemma is crucial to prove the correctness of algorithm \FTF and it explains the main consequences of algorithm \TM.

\begin{prop}\label{prop:TripleMax}
Let $p\in\MAX$ and let $f$ be a face of type II or type III for $p$. If we call $\TM(p)$ then
\begin{enumerate}[label=\theprop.(\arabic*), ref=\theprop.(\arabic*),leftmargin=\widthof{10.(1)}+\labelsep]\itemsep0em
\item\label{item:prop_triplemax_1} let $m\in\Max_p$ interfere with $f$, then $triple(m)\cap triple(p)=\emptyset$,
\item\label{item:prop_triplemax_2} let $m,m'\in\Max_p$ interfere with $f$, then $triple(m)\cap triple(m')=\emptyset$.
\end{enumerate}
\end{prop}
\begin{proof}
The first statements is implied by Line~\ref{line:15_definition:X_Y_W_Z} and all lines in which a value of $triple$ is assigned. To prove the second statement let $f$ and $g$ be a face of type II and type III for $p$, respectively. By definition, the unique paths in $\Max_p$ that interfere with $f$ are $m^f_r$ and $m^f_\ell$. The for cycle in Line~\ref{line:15_for_II_type} implies $triple(m^f_r)=X$ and $ triple(m^f_\ell)=Y$, and the thesis applies in this case. Now let $m,m'$ interfere with $g$. Hence either $m\dashrightarrow m'$ or $m'\dashrightarrow m$ and w.l.o.g.~we assume that $m'\dashrightarrow m$. Thus $g=f_m$ and there are two cases: either $m\in\Max^{II}_p$ or $m\in\Max^{III}_p$. If the former case applies, then $triple(m)\in\{X,Y\}$ because of the for cycle in Line~\ref{line:15_for_II_type} and $triple(m)\in\{W,Z\}$ because of  Line~\ref{line:15_A} and Line~\ref{line:15_B}, so the thesis holds in this case. If the latter case applies, then Line~\ref{line:15_bipartite}, Line~\ref{line:15_A} and Line~\ref{line:15_B} imply $triple(m)=W$ and $triple(m')=Z$, or vice-versa, and thus the thesis holds. We stress that we can bipartite the graph $(\Max_p^{III},\dashrightarrow)$ because Lemma~\ref{lemma:forest} states that it is a forest. 
\end{proof}

\subsection{Correctness of algorithm \FTF}\label{sec:15_FTF}

In this subsection we prove the correctness of algorithm \FTF shown in Subsection~\ref{sec:15_outline}. As a consequence we have Corollary~\ref{cor:15_forest_teorico} which state that every \NCS has \FCN at most 15.

\begin{theorem}\label{th:15_forest}
Given a \NCS $P$, algorithm $\FTF$ produces a forest labeling $\mathcal{L}$ of $P$ which uses at most 15 labels.
\end{theorem}
\begin{proof}
Thanks to Theorem \ref{th:faces} we only need to prove that every faces in $\bigcup_{q\in P}q$ is solved by $\mathcal{L}$. Let $f$ be a face, then there exists $p\in\MAX$ such that $f$ is a face of type I, or type II or type III for $p$.

If $f$ is of type I for $p$, then $f$ is solved by $\mathcal{L}$ because of Proposition~\ref{prop:Color} and the call $\CT(p)$.

If $f$ is of type II for $p$, then $e^f_r\in m^f_r$ and $e^f_\ell\in m^f_\ell$. Moreover, if a path $q\in P$ interferes with $f$, then either $q\in\Touch_{m^f_r}$ or $q\in\Touch_{m^f_\ell}$. Thus $f$ is solved by $\mathcal{L}$ because of~\ref{item:prop_triplemax_2} and the calls $\CT(m^f_r)$ and $\CT(m^f_\ell)$.

If $f$ is a face of III type for $p$, then either $e^f_r\in m^f$ or $e^f_\ell\in m^f$, where $m^f$ is the unique child of $q$ belonging to $\Max_p$. W.l.o.g., we assume that $e^f_r\in m^f$. We observe that $e^f_\ell$ may belong to some paths in $\Touch_p$ and at most one path in $\Max_p$. Every path $q$ in $\Touch_p$ satisfies $\mathcal{L}(q)\in triple(p)$, thus thanks to~\ref{item:prop_triplemax_1} and the call $\CT(m^f)$ we can ignore paths in $\Touch_p$ to determine if $f$ is solved by $\mathcal{L}$. Hence let us assume that there exists $m'\in\Max_p$ satisfying $e^f_\ell\in m'$. Then $f$ is solved by $\mathcal{L}$ because of both the statements of Proposition~\ref{prop:TripleMax} and the calls $\CT(m^f)$ and $\CT(m')$.
\end{proof}

\begin{corollary}\label{cor:15_forest_teorico}
Let $P$ be a set of non-crossing shortest paths in a plane graph $G$ whose extremal vertices lie on the external face of $G$. Then the \FCN of $P$ is at most 15.
\end{corollary}

\section{The \FCN is at most 4}\label{sec:4_ALL}

In this section we show that the \FCN of a \NCS $P$ is at most 4. In particular, we present algorithm \FF that produces a forest labeling of $P$ which uses at most 4 labels. We strictly use all results in Section~\ref{sec:15_ALL}.

In Subsection~\ref{sec:4_outline} we describe the outline of algorithm $\FF$. In Subsection~\ref{sec:4_type_I} we deal with faces of the first type and Subsection~\ref{sec:4_type_II_III} with faces of second and third type. Finally, in Subsection~\ref{sec:4_FF} we exhibit algorithm $\FF$ and we prove its correctness.

\subsection{Outline of the algorithm}\label{sec:4_outline}

We stress that algorithm \FF has the same structure of algorithm \FTF, the only differences are that \CT and \TM are replaced with \CST and \TSM, respectively. We introduce $\scc(p)\in triple(p)$ for all $p\in\MAX$ as a special label of $triple(p)$. Algorithm $\CST(p)$ assigns one label of $triple(p)$ for all $q\in\Touch_p$ so that all paths containing $\se^p$ (see Definition~\ref{def:special_edge}) are labeled with $\scc(p)$. Algorithm $\TSM(p)$ assigns $triple(q)$ and $\scc(q)$ for all $q\in\Max_p$. The assignments are set so that at iteration $i$ all faces in $\Touch_p\cup\Max_p$ are solved by $\mathcal{L}$, for all $p\in\MAX_i$.

\begin{figure}[h]
\begin{algorithm}[H]
\SetAlgorithmName{$\FF$}{}{}
\renewcommand{\thealgocf}{}
 \caption{}
  \KwIn{a \NCS $P$}
 \KwOut{a forest labeling $\mathcal{L}:P\mapsto [4]$}
{Transform $P$ from a \NCS to a \BNCS\;
For each path $p\in P$ define the global variable $\mathcal{L}(p)$ initialized to \NULL\;
For each path $p\in\MAX$ define the global variables $triple(p),\scc(p)$ both initialized to \NULL\;
$triple(p_1)\leftarrow \{1,2,3\}$\;
$\scc(p_1)\leftarrow 1$\;
\For{$i=0,\ldots,N$\label{line:4_cycle1}}{
\For{{\normalfont\textbf{each}} $p\in\MAX_i$}{
$\CST(p)$\tcp*{assign $\mathcal{L}(q)\in triple(p)$ for all $q\in\Touch_p$}
$\TSM(p)$\tcp*{assign $triple(q),\scc(q)$ for all $q\in\Max_p$}
}
}
\Return $\mathcal{L}$\;
}
\end{algorithm}
\end{figure}

\subsection{Dealing with faces of type I}\label{sec:4_type_I}

The main goal of this subsection is to build algorithm $\CST(p)$ which assigns $\mathcal{L}(q)\in triple(p)$ for all $q\in\Touch_p$ so that all paths containing $\se^p$ are labeled with $\scc(p)$ and all faces of type I for $p$ are solved by $\mathcal{L}$ (see Proposition~\ref{prop:ColorSpecial}), where $p\in\MAX$. We obtain algorithm \CST by joining algorithms \RC and \LC.

The algorithm \LC is equal to \RC except for the last line in which $q_r$ and $q_\ell$ are swapped. Therefore, Lemma~\ref{lemma:LC} is a consequence of Lemma~\ref{lemma:RC}.

\begin{figure}[h]
\begin{algorithm}[H]
\SetAlgorithmName{$\LC(p,q,( c_1,c_2,c_3) )$}{}{}
\renewcommand{\thealgocf}{}
\caption{}
 \KwIn{a path $p$ in $\MAX$, a path $q$ in $Touch_p$ and a permutation $( c_1,c_2,c_3)$ of $triple(p)$}
  \KwOut{assign $\mathcal{L}(q')\in triple(p)$ for all $q'\in\Touch_p$ such that $q'\preceq q$}
{$\mathcal{L}(q)\leftarrow c_1$\;
\If{$q$ has one child $q'$ in $\Touch_p$}{$\LC(p,q',( c_1,c_2,c_3))$\;}
\If{$q$ has two children in $\Touch_p$}{
$\LC(p,q_\ell,( c_1,c_3,c_2))$\; $\LC(p,q_r,( c_2,c_1,c_3))$\;}
}
\end{algorithm}
\end{figure}

\begin{figure}[h]
\captionsetup[subfigure]{justification=centering}
\centering
\begin{subfigure}{7cm}
\begin{overpic}[width=7cm,percent]{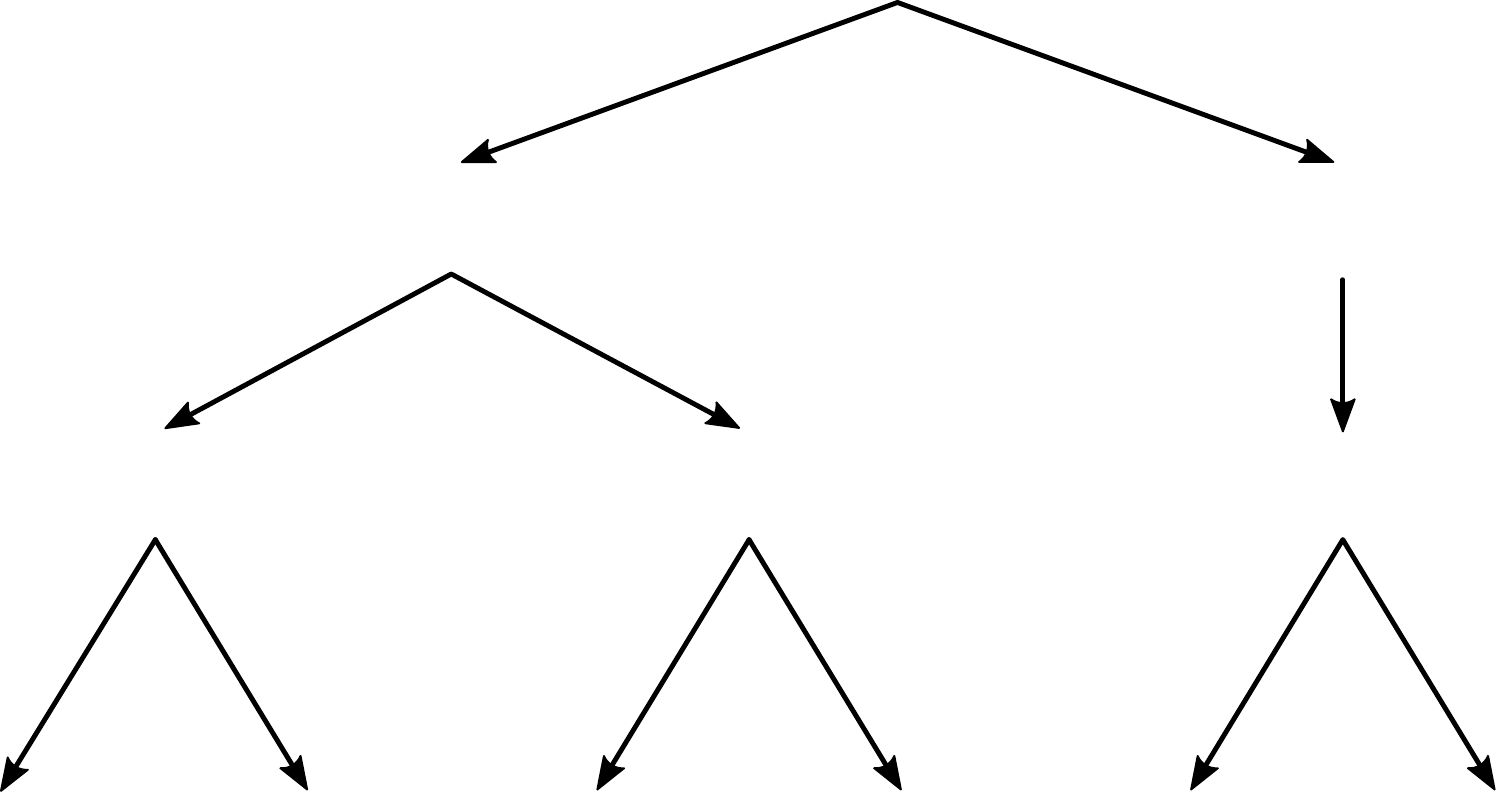}
\put(51.3,55.5){$( {\color{red} 1},2,3)$}

\put(21.6,37){$( {\color{red} 1},3,2)$}
\put(81.1,37){$( {\color{red} 2},1,3)$}

\put(1.8,19){$( {\color{red} 1},2,3)$}
\put(41.4,19){$( {\color{red} 3},1,2)$}
\put(81.1,19){$( {\color{red} 2},1,3)$}

\put(-8.6,-5){$( {\color{red} 1},3,2)$}
\put(11.8,-5){$( {\color{red} 2},1,3)$}
\put(31,-5){$( {\color{red} 3},2,1)$}
\put(51.5,-5){$( {\color{red} 1},3,2)$}
\put(70.9,-5){$( {\color{red} 2},3,1)$}
\put(91.4,-5){$( {\color{red} 1},2,3)$}

\put(-6,55.5){\LC}
\end{overpic}
\end{subfigure}
\vspace{2mm}
\caption{assume that $triple(p)=\{1,2,3\}$. Example of permutations of $triple(p)$ during the call $\LC(p,q,( 1,2,3))$, where $q$ is the vertex on the top. The first element of the permutation, in red, denotes also the label assigned to the path.}\label{fig:LeftColor}
\end{figure}

\begin{lemma}\label{lemma:LC}
Let $p\in\MAX	$, let $( c_1,c_2,c_3)$ be any ordering of $triple(p)$ and let $q\in\Touch_p$. If we call $\LC(p,q,( c_1,c_2,c_3))$, then $\mathcal{L}(R(q))\subseteq\{c_1,c_2\}$ and $\mathcal{L}(L(q))=\{c_1\}$.
\end{lemma}

We can use algorithm \RC and algorithm \LC to label all elements in $\Touch_p$ in order to solve all faces of type I for $p$ and to assure that all paths containing $\se^p$ are labeled with $\scc(p)$. This is made with algorithm \SC whose entries are $p,q,( c_r,c_\ell)$, where $p$ is a path in $\MAX$, $q$ is a path in $\Touch_p$, and $c_r$, $c_\ell$ are two labels in $triple(p)$. 

We call $\SC(p,q,( c_r,c_\ell))$ if and only if $q$ contains $\se^p$, and the child of $q$ containing $\se^p$ is labeled with $\scc(p)$; moreover, if the right (resp. left) child of $q$ does not contain $\se^p$, then we label it with $c_r$ (resp., $c_\ell$). If both children of $q$ do not contain $\se^p$, then we label all left descendants of $q$ with $\scc(p)$---in a way, we assume that if the right child of $q$ does not contain $\se^p$, then the left child does. Clearly, if $q$ has one child in $Touch_p$, then we assume that it contains $\se^p$.

It holds that $c_r,c_\ell\in triple(p)\setminus\{\scc(p)\}$ but they are not necessarily distinct. Let $Q$ be the set of all elements in $\Touch_p$ that contain $\se^p$. We observe that $Q$ forms a path in $\text{Tree}\Touch_p$. In few words, algorithm \SC calls \RC for paths on the left of $Q$ and \LC for paths on the right of $Q$ (see Figure \ref{fig:setcolor}). The recursive calls of \SC turn the labels in order to obtain the following lemma.

\begin{figure}[h]
\captionsetup[subfigure]{justification=centering}
\centering
%FIGURA 1
	\begin{subfigure}{13cm}
\begin{overpic}[width=13cm,percent]{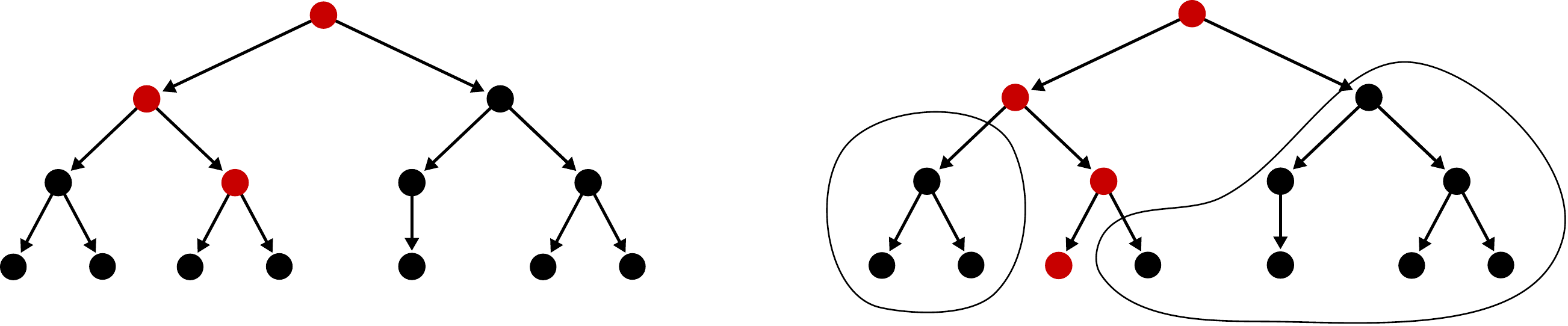}
\put(-2,18.5){$TreeTouch_p$}
\put(48,14.5){\RC}
\put(85.5,17.6){\LC}
\end{overpic}
%\caption{Esempi di facce di I, II e III tipo}\label{LABEL1}
\end{subfigure}
  \caption{on the left $TreeTouch_p$ in which red vertices correspond to paths containing $\se^p$. On the right, red vertices correspond to paths labeled with $\scc(p)$. We show for which vertices \SC calls \RC and for which \LC.}
\label{fig:setcolor}
\end{figure}

\begin{figure}[h]
\begin{algorithm}[H]
\SetAlgorithmName{$\SC(p,q,( c_r,c_\ell) )$}{}{}
\renewcommand{\thealgocf}{}
 \caption{}
  \KwIn{a path $p\in\MAX$, a path $q\in Touch_p$, two labels $c_r,c_\ell\in triple(p)\setminus\{\scc(p)\}$}
 \KwOut{assign $\mathcal{L}(q')\in triple(p)$ for all $q'\in\Touch_p$ such that $q'\preceq q$ so that $\mathcal{L}(q'')=\scc(p)$ for all $q''\in\Touch_p$ containing $\se^p$}
{$\mathcal{L}(q)\leftarrow \scc(p)$\label{line:color_Mix}\;
\lIf{$q$ has one child $q'$ in $\Touch_p$}{$\SC(p,q',( c_r,c_\ell) )$}\label{line:one_child_Mix}
\If{$q$ has two children in $\Touch_p$}{
\uIf{$\se^p\in q_r$\label{line:call_right_Mix}}{
$\SC(p,q_r,( c_r,triple(p)\setminus\{c_\ell,\scc(p)\}) )$\label{line:call_right_Mix_1}\; 
$\RC(p,q_\ell,( c_\ell,\scc(p),triple(p)\setminus\{c_\ell,\scc(p)\}) )$\label{line:call_right_Mix_2}\;
}
\Else{\label{line:call_left_Mix}
$\SC(p,q_\ell,( triple(p)\setminus\{c_r,\scc(p)\},c_\ell) )$\;
$\LC(p,q_r,( c_r,\scc(p),triple(p)\setminus\{c_r,\scc(p)\}) )$\;
}
}
%\Return $\mathcal{L}$\;
}
\end{algorithm}
\end{figure}

\begin{lemma}\label{lemma:SC}
Let $p\in\MAX$, let $c_r,c_\ell\in triple(p)\setminus\{\scc(p)\}$ be not necessarily distinct and let $q\in\Touch_p$ satisfy $\se^p\in q$. If we call $\SC(p,q,( c_r,c_\ell))$, then $\mathcal{L}(R(q))\subseteq\{c_r,\scc(p)\}$ and $\mathcal{L}(L(q))\subseteq\{c_\ell,\scc(p)\}$.
\end{lemma}
\begin{proof}
If $q$ has one child in $TreeTouch_p$, then algorithm \SC does the same call to child, thus we can assume that $q$ has two children in $TreeTouch_p$. W.l.o.g., we assume that $\se^p\in q_r$, indeed, if $\se^p\not\in q_r$, then the proof is symmetric. Thus algorithm \SC does the call $\SC(p,q_r,( c_r,triple(p)\setminus\{c_\ell,\scc(p)\}))$ and the call $\RC(p,q_\ell,$ $( c_\ell,\scc(p),triple(p)\setminus\{c_\ell,\scc(p)\})$. The second call implies that $\mathcal{L}(L(q))\subseteq\{c_\ell,\scc(p)\}$ by Lemma \ref{lemma:RC}. It remains to show that $\mathcal{L}(R(q))\subseteq\{c_r,\scc(p)\}$.

Let $w$ be the first element w.r.t.\ $\preceq$ in $R(q)$ such that $w_r$ does not contain $\se^p$, i.e., $\se^p\not\in w_r$ and $\se^p\in z$ for all $z\in R(q)\setminus R(w_r)$. Thus all elements in $R(q)$ before (w.r.t.\ $\preceq$) $w_r$ are labeled with $\scc(p)$ because of Line \ref{line:color_Mix}. Moreover, by Line \ref{line:call_left_Mix} we call $\LC(p,w_r,( c_r,\scc(p),$ $triple(p)\setminus\{c_r,\scc(p)\}))$; indeed we note that the third element (i.e., $c_r$) in all these recursive calls of \SC does not change because of Line~\ref{line:call_right_Mix_1}. Finally, Lemma \ref{lemma:LC} implies that $\mathcal{L}(R(w_r))\subseteq \{c_r,\scc(p)\}$, thus $\mathcal{L}(R(q))\subseteq\{c_r,\scc(p)\}$ as we claimed.
\end{proof}

\begin{prop}\label{prop:ColorSpecial}
Let $p\in\MAX$, if we call $\CST(p)$, then every face of type I for $p$ is solved by $\mathcal{L}$, $\mathcal{L}(q)\in triple(p)$ for all paths in $\Touch_p$ and $\mathcal{L}(q')=\scc(p)$ for all paths in $\Touch_p$ containing $\se^p$.
\end{prop}
\begin{proof}
By the recursion of algorithm $\CST$, it is clear that all paths in $\Touch_p$ are labeled with labels in $triple(p)$ and all paths in $\Touch^p$ containing $\se^p$ are labeled with $\scc(p)$. It remains to prove that every face of type I for $p$ is solved by $\mathcal{L}$. We need the following preliminary claim.
\begin{enumerate}[label=\alph{CONT})]\stepcounter{CONT}
\item\label{claim:Color_face_set} Let $c_r,c_\ell\in triple(p)$ be not necessarily distinct and let $f$ be a face of type I for $p$ such that the upper path $q$ of $f$ is in $\Touch_p$. If we call $\SC(p,q,( c_r,c_\ell))$, then $f$ is solved by $\mathcal{L}$.
\end{enumerate}

\claimbegin{\ref{claim:Color_face_set}} It holds that $q$ contains $\se^p$ and one of its children contains $\se^p$ (we have assumed that if the right child does not contain $\se^p$, then the left child does). W.l.o.g., $q_r$ contains $\se^p$. Thus $\mathcal{L}(q_r)=\scc(p)$, $\mathcal{L}(q_\ell)=c_\ell$, we call $\SC(p,q_r,( c_r,triple(p)\setminus\{c_\ell,\scc(p)\}))$ and $\RC(p,q_\ell,( c_\ell,\scc(p),triple(p)\setminus\{c_\ell,\scc(p)\}))$.

By Lemma~\ref{lemma:right_left}, $\mathcal{L}(e^f_r)\in \mathcal{L}(L(q_r))$ and $\mathcal{L}(e^f_\ell)\in \mathcal{L}(R(q_\ell))$. By Lemma \ref{lemma:RC} and by the call $\RC(p,q_\ell,( c_\ell,\scc(p),triple(p)\setminus\{c_\ell,\scc(p)\}))$, it holds that $\mathcal{L}(R(q_\ell))\subseteq\{c_\ell\}$. 
Moreover, by calling $\SC(p,q_r,( c_r,triple(p)\setminus\{c_\ell,\scc(p)\}))$ and by Lemma~\ref{lemma:SC}, $\mathcal{L}(L(q_r))\subseteq\{(triple(p)\setminus\{c_\ell,\scc(p)\}),\scc(p)\}= triple(p)\setminus\{c_\ell\}$. Hence $f$ is solved by $\mathcal{L}$ because $\mathcal{L}(e^f_r)\cap \mathcal{L}(e^f_\ell)=\emptyset$.
\claimend{\ref{claim:Color_face_set}}

Let $f$ be a face of type I for $p$ and let $q$ be its upper path. The call $\CST(p)$ implies either $\RC(p,q,( c_1,c_2))$ or $\LC(p,q,( c_1,c_2))$ or $\SC(p,q,(c_r,c_\ell))$ for some distinct $c_1,c_2\in triple(p)$ and not necessarily distinct $c_r,c_\ell\in triple(p)\setminus\{\scc(p)\}$. If one among the first two cases applies, then $f$ is solved by $\mathcal{L}$ by applying Lemma~\ref{lemma:RC} or Lemma~\ref{lemma:LC} and the same reasoning of Proposition~\ref{prop:Color}'s proof. If the last case applies, then $f$ is solved by $\mathcal{L}$ because of~\ref{claim:Color_face_set}.
\end{proof}

\begin{figure}[h]
\begin{algorithm}[H]
\SetAlgorithmName{$\CST(p)$}{}{}
\renewcommand{\thealgocf}{}
\caption{}
 \KwIn{a path $p$ in $\MAX$}
 \KwOut{assign $\mathcal{L}(q)\in triple(p)$ for all $q\in\Touch_p$ so that every face of type~I for $p$ is solved by $\mathcal{L}$ and $\mathcal{L}(q')=\scc(p)$ for all $q'\in\Touch_p$ containing $\se^p$}
{Let $c_r,c_\ell$ in $triple(p)\setminus\{\scc(p)\}$ be not necessarily distinct\;
\SC($p,p,( c_r,c_\ell))$\tcp*{assign $\mathcal{L}(q)\in triple(p)$ for all $q\in\Touch_p$}
%\Return $\mathcal{L}$\;
}
\end{algorithm}
\end{figure}

The following corollary is a consequence of Lemma~\ref{lemma:RC}, Lemma~\ref{lemma:LC} and Lemma~\ref{lemma:SC} and it is crucial in the proof of Lemma~\ref{lemma:T_two_colors} in the next subsection.

\begin{corollary}\label{cor:R(q)_or_L(q)}
Let $p\in\MAX$. If we call $\CST(p)$, then $|\mathcal{L}(R(q))|\leq2$ and $|\mathcal{L}(L(q))|\leq2$ for all $q\in\Touch_p$.
\end{corollary}

\subsection{Dealing with faces of type II and type III}\label{sec:4_type_II_III}

In algorithm \FTF we use many labels, and we assign $triple(q)$ for each $q\in\Max_p$ so that $triple(q)\cap triple(p)=\emptyset$, ignoring paths in $\Touch_p$ because they are labeled with labels in $triple(p)$, where $p\in\MAX$. In this section, arguing with algorithm \FF, we have only four labels available, and the previous request is impossible to satisfy; indeed two distinct triples in $\{1,2,3,4\}$ have non empty intersection. Thus we have to consider also paths in $\Touch_p$.  In algorithm \TSM we assign $triple(q)$ for each $q\in\Max_p$ by visiting each tree $T$ in $(\Max_p^{III},\dashrightarrow)$, and in the following lemma we show that all paths in $\Touch_p$ that interfere with a face related to $T$ has at most two labels.

\begin{lemma}\label{lemma:T_two_colors}
Let $p\in\MAX$, let $T$ be a tree of the forest $(\Max_p^{III},\dashrightarrow)$ and let $\Delta=C\big(\{q\in\Touch_p \,|\, q$ interferes with $f_m$ for some $m\in V(T)\}\big)$. If we call $\CST(p)$, then $|\Delta|\leq2$.
\end{lemma}
\begin{proof}
For convenience, let $Q=\{q\in\Touch_p \,|\, q$ interferes with $f_m$ for some $m\in V(T)\}$. Thanks to Corollary~\ref{cor:R(q)_or_L(q)}, it suffices to prove that either $Q\subseteq R(q)$ or $Q\subseteq L(q)$ for some $q\in\Touch_p$. For each $m\in V(T)$ we denote by $q_m$ its parent in $T_g$. Let $m^{\text{root}}$ be the root of $T$ and, w.l.o.g., we assume that $m^{\text{root}}$ is the right child in $T_g$ of $q_{m^\text{root}}$. For the sake of clarity the proofs of~\ref{item:all_right_child}, \ref{item:all_right_descendant} and \ref{item:path} are at the end of the main proof. 

\begin{enumerate}[label=\alph{CONT})]\stepcounter{CONT}
\item\label{item:all_right_child} let $m,m'\in V(T)$ satisfy $m\dashrightarrow m'$. If $m$ is the right child of $q_m$ in $T_g$, then $m'$ is the right child of $q_{m'}$ in $T_g$.
\end{enumerate}

Because of \ref{item:all_right_child} and being $m^{\text{root}}$ the right child of $q_{m^\text{root}}$, given any $m\in V(T)$, then $m$ is the right child of its parent $q_m$ in $T_g$. For each $m\in V(T)$ we define $P_m$ as the path in $V(T)$ from the root $m^{\text{root}}$ to $m$. Moreover, for each $m\in V(T)$, let $Q_m=\{q\in\Touch_p \,|$ $q$ interferes with $f_m\}$ and $Q_{P_m}=\bigcup_{m'\in P_m}Q_{m'}$. We need other two statements:

%\todob{lettere f,g messe con vspace...}
\begin{enumerate}[label=\alph{CONT})]\stepcounter{CONT}   
\item\label{item:all_right_descendant} for each $m\in V(T)$ it holds that $Q_m\subseteq R(q_m)$,
\end{enumerate}
\vspace*{-6mm}
\begin{enumerate}[label=\alph{CONT})]\stepcounter{CONT}
\item\label{item:path} for each $m\in V(T)$ it holds that $Q_{P_m}\subseteq R(q_m)$.
\end{enumerate}

Let $w^{\text{root}}\in\Touch_p$ be the child of $q_{m^\text{root}}$ different from $m^{\text{root}}$. Being $m^{\text{root}}$ in $\Max_p^{III}$ then $w^{\text{root}}$ interferes with $f_{m^{\text{root}}}$, hence, $w^{\text{root}}\in\bigcap_{m\in V(T)}Q_{P_m}$. By~\ref{item:path}, $w^{\text{root}}\in R(q_m)$ for every $m\in V(T)$, consequently we can order all elements in $V(T)$ in $(m_1,m_2,\ldots,m_{|V(T)|})$ so that $q_{m_1}\in R(q_{m_2})\subseteq R(q_{m_3})\subseteq\ldots\subseteq R(q_{m_{|V(T)|)}})$;  otherwise there would be a vertex in $TreeTouch_p$ with two incoming darts because of the definition of right descendant in Definition~\ref{def:R(q)_L(q)}. It is easy to see that $m_{|V(T)|}$ is a leaf in $T$. Finally, $Q=\bigcup_{m\in V(T)}Q_{m}\subseteq\bigcup_{m\in V(T)}Q_{P_m}\subseteq R(q_{m_{|V(T)|}})$ by above reasoning and~\ref{item:path}, and the thesis holds. To complete the proof, now we prove the previous three claims.

%\todob{ricontrollare queste proof}
\vspace*{2mm}
\claimbegin{\ref{item:all_right_child}} let $w$ be the child of $q_m$ different from $m$ and let $w'$ be the child of $q_{m'}$ different from $m'$. By the single-touch property, $q_{m'}$ splits $w'$ into three parts: $w'=r\circ q_{m'}\cap w'\circ \ell$, with $x_{w'}\in r$ and $y_{w'}\in\ell$. Let $e_r$ (resp., $e_\ell$) be the extremal edge of $r$ (resp., $\ell$) that does not contain $x_{w'}$ (resp., $y_{w'}$). All these paths and edges are shown in Figure~\ref{fig:right_right}.

Being $m\dashrightarrow m'$, then either $e_r\in m$ or $e_\ell\in m$; indeed one among $e_r$ and $e_\ell$ is an extremal edge of the lower boundary of $f_{m'}$. By Remark~\ref{remark:crucial} applied to $w'$ and $q_{m'}$, the thesis holds if $e_r\in m$. We note that if $e_\ell\in m$, then $w_m$ does not intersect $p$ because of Remark~\ref{remark:crucial} and because $m$ is the right child of $q_m$ in $T_g$. This is absurdum because $w_m\in\Touch_p$. Thus, by above, $e_r\in m$ and the thesis is proved.
\claimend{\ref{item:all_right_child}}

\claimbegin{\ref{item:all_right_descendant}} being $m$ the right child of $q_m$, if $q_{m'}\in\Touch_p$ interferes with $f_m$, then $e^{f_m}_r$ belongs to $q_{m'}$. Thus all paths in $Q_m$ contain $e^{f_m}_r$ and they can be ordered w.r.t.\ $\preceq$. Hence $Q_m=\{q_1,q_2,\ldots,q_z\}$ so that $q_{i+1}\preceq q_i$ for all $i\in[z-1]$. Now we can proceed by induction on $i$. We note that $q_1$ is the child of $q_m$ in $T_g$ different from $m$ and $q_z\in R(q_m)$ because $q_z$ is the only child in $TreeTouch_p$. Therefore, we have proved the base case. The induction case is trivial if $q_i$ has only one child in $TreeTouch_p$, otherwise it follows from the same reasoning of \ref{item:all_right_child}.
\claimend{\ref{item:all_right_descendant}}

\claimbegin{\ref{item:path}} let $(m_1,m_2,\ldots,m_w)$ be the ordered sequence of vertices in $P_m$. Then $Q_{m_i}\subseteq R(q_{m_i})$ because of~\ref{item:all_right_descendant}. Moreover, being $m_i\dashrightarrow m_{i+1}$, then $q_{m_i}\in Q_{m_{i+1}}$ (this is formally proved in~\ref{claim_1} of Lemma~\ref{lemma:forest}). Hence for any $i\in[w]$, $Q_{m_i}\in R(q_{m_w})$ by applying repeatedly as stated, and the thesis follows.
\claimend{\ref{item:path}}

The proof is now complete.
\end{proof}

\begin{figure}[h]
\captionsetup[subfigure]{justification=centering}
\centering
%FIGURA 1
	\begin{subfigure}{5cm}
\begin{overpic}[width=5cm,percent]{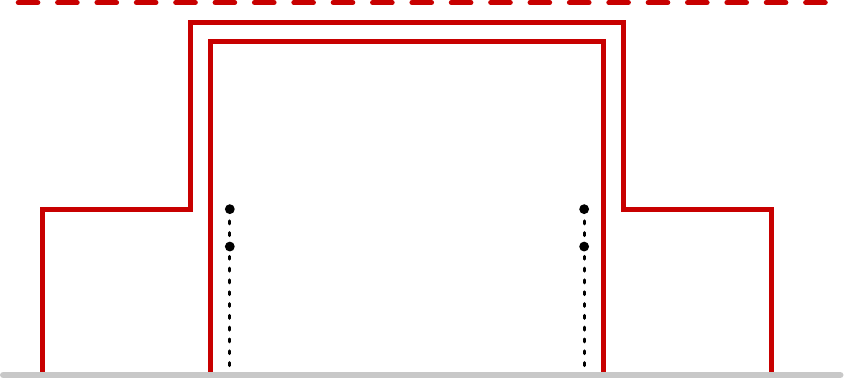}
\put(-4,43){$p$}
\put(-8,0){$f^\infty$}

\put(46,32){$w'$}
\put(79,14){$q_{m'}$}
\put(63,4){$r$}
\put(29,4){$\ell$}

\put(61,16){$e_r$}
\put(29,16){$e_\ell$}

\end{overpic}
%\caption{Esempi di facce di I, II e III tipo}\label{LABEL1}
\end{subfigure}
  \caption{paths and edges used in the proof of~\ref{item:all_right_child}. In red continuous line paths $w'$ and $q_{m'}$, in black dotted lines path $r$ and $\ell$, and we point out edges $e_r$ and $e_\ell$.}
\label{fig:right_right}
\end{figure}

%We can now exhibit algorithm \TSM.

%\todob{se c'e' spazio introdurre l'algoritmo}

\begin{figure}[h]
\begin{algorithm}[H]
\SetAlgorithmName{$\TSM(p)$}{}{}
\renewcommand{\thealgocf}{}
 \caption{}
  \KwIn{a path $p$ in $\MAX$}
  \KwOut{an assignment of $triple(q)$  and $\scc(q)$ for all $q\in\Max_p$ }
{\For{{\normalfont\textbf{each}} face $f$ of type II for $p$\label{line:4_for_type_II}}{
Let $\textit{label}_r,\textit{label}_\ell\in triple(p)$ be distinct\;
$\scc(m^f_r)\leftarrow \textit{label}_r$\;
$\scc(m^f_\ell)\leftarrow \textit{label}_\ell$\;
$triple(m^f_r)\leftarrow triple(p)$\label{line:4_triple(p)_1}\;
$triple(m^f_\ell)\leftarrow triple(p)$\label{line:4_triple(p)_2}\;
}
\For{{\normalfont\textbf{each}} tree $T$ in the forest $(\Max_p^{III},\dashrightarrow)$\label{line:4_for_tree_T}}{
Let $\Delta=C\big(\{q\in\Touch_p \,|\, q$ interferes with $f_m$ for some $m\in V(T)\}\big)$\label{line:Delta}\;
If $|\Delta|=1$, then add to $\Delta$ another label of $triple(p)$\label{line:Delta+1}\; 
Bipartite the tree $T$ into the two classes $A$ and $B$ so that the root of $T$ is in $A$ and $u\not\dashrightarrow v$ for all $u,v$ in the same class\label{line:4_bipartite}\;
Let $\textit{label}_{A}=\{1,2,3,4\}\setminus\{triple(p)\}$ and $\textit{label}_{B}=triple(p)\setminus \Delta$\label{line:c_A_c_B}\;
\For{{\normalfont\textbf{each}\label{line:cycle_A}} $m\in A$}{
$\scc(m)\leftarrow \textit{label}_{A}$\;
$triple(m)\leftarrow \Delta\cup \{\textit{label}_{A}\}$\label{line:A}}
\For{{\normalfont\textbf{each}} $m\in B$\label{line:cycle_B}}{
$\scc(m)\leftarrow \textit{label}_{B}$\;
$triple(m)\leftarrow triple(p)$\label{line:B}}
}
%\Return $triple(m)$ and $\scc(m)$ for all $m\in\Touch_p$\;
}
\end{algorithm}
\end{figure}

The following result explains the effects of algorithm \TSM. We recall that, given $m\in\MAX\setminus\{p_1\}$, $f_m\in\mathcal{F}$ denotes the face whose upper path is the parent of $m$ in $T_g$.

\begin{prop}\label{prop:TripleSpecialMax}
Let $p\in\MAX$, let $f$ be a face of type II for $p$, let $g$ be a face of III for $p$ and let $m\in\Max^{III}_p$ satisfy $g=f_m$. If we call $\CST(p)$ and $\TM(p)$ then
\begin{enumerate}[label=\theprop.(\arabic*), ref=\theprop.(\arabic*),leftmargin=\widthof{10.(1)}+\labelsep]\itemsep0em
\item\label{item:special_1} $\scc(m^f_r)\neq\scc(m^f_\ell)$,
\item\label{item:special_2} let $q\in\Touch_p$ interfere with $g$. Then $\scc(m)\neq \mathcal{L}(q)$,
\item\label{item:special_3} let $m'\in\Max_p\setminus\{m\}$ interfere with $g$. Then $\scc(m)\not\in triple(m')$.
\end{enumerate}
\end{prop}
\begin{proof}
The first statement is a consequence of the for cycle in Line~\ref{line:4_for_type_II}. Let $T$ be the tree in $(\Max_p^{III},\dashrightarrow)$ containing $m$ and let $\Delta$ be as defined in Line~\ref{line:Delta}. The second statement is implied by the for cycle in Line~\ref{line:4_for_tree_T}, indeed, $\mathcal{L}(q)\in\Delta$ and $\scc(m)\not\in\Delta$ because of Line~\ref{line:c_A_c_B}.

It remains to prove the third statement. There are two cases: $m'\in\Max_p^{II}$ and $m'\in\Max_p^{III}$. If the former case applies, then $m$ is the root of $T$. Thus $triple(m)=triple(p)$ because of Line~\ref{line:4_triple(p)_1} or Line~\ref{line:4_triple(p)_2} and $\scc(m)\not\in triple(p)$ because of Line~\ref{line:4_bipartite} that assigns the class $A$ to the root of $T$. If the latter case applies, then $m'\dashrightarrow m$. Hence, because of Line~\ref{line:4_bipartite}, $m\in A$ and $m'\in B$, or vice-versa. Finally,  $\scc(m)\not\in triple(m')$ because of the for cycles in Line~\ref{line:cycle_A} and Line~\ref{line:cycle_B} (it is only a fact of checking).
\end{proof}

\subsection{Correctness of algorithm \FF}\label{sec:4_FF}

We prove in Theorem~\ref{th:4_forest} the correctness of algorithm \FF shown in Subsection~\ref{sec:4_outline}; the proof is analogous to Theorem~\ref{th:15_forest}'s proof. As a consequence we have Corollary~\ref{cor:4_forest_teorico} which state that every \NCS has \FCN at most 4.

\begin{theorem}\label{th:4_forest}
Given a \NCS $P$, algorithm $\FF$ produces a forest labeling $\mathcal{L}$ of $P$ which uses at most 4 labels.
\end{theorem}
\begin{proof}
Thanks to Theorem \ref{th:faces} we only need to prove that every faces in $\bigcup_{p\in P}p$ is solved by $\mathcal{L}$. Let $f$ be a face, then there exists $p\in\MAX$ such that $f$ is a face of type I, or type II or type III for $p$.

If $f$ is of type I for $p$, then $f$ is solved by $\mathcal{L}$ because of the call $\CST(p)$ and Proposition~\ref{prop:ColorSpecial}.

If $f$ is of type II for $p$, then $e^f_r\in m^f_r$ and $e^f_\ell\in m^f_\ell$. Moreover, if a path $q\in P$ interferes with $f$, then either $q\in\Touch_{m^f_r}$ or $q\in\Touch_{m^f_\ell}$. Thus $f$ is solved by $\mathcal{L}$ because of~\ref{item:special_1} and the calls $\CST(m^f_r)$ and $\CST(m^f_\ell)$.

If $f$ is a face of III type for $p$, then either $e^f_r\in m^f$ or $e^f_\ell\in m^f$, where $m^f$ is the unique child of $q$ belonging to $\Max_p$. W.l.o.g., we assume that $e^f_r\in m^f$. We observe that $e^f_\ell$ may belong to some paths in $\Touch_p$ and at least one in $\Max_p$. If there does not exist any path $m'\in\Max_p$ satisfying $e^f_\ell\in m'$, then $f$ is solved by $\mathcal{L}$ because of~\ref{item:special_2} and the call $\CST(m^f)$.  Otherwise, let $m'\in\Max_p$ satisfy $e^f_\ell\in m'$. Then $f$ is solved by $\mathcal{L}$ because of~\ref{item:special_2}, \ref{item:special_3} and the calls $\CST(m^f)$ and $\CST(m')$.
\end{proof}

\begin{corollary}\label{cor:4_forest_teorico}
Let $P$ be a set of non-crossing shortest paths in a plane graph $G$ whose extremal vertices lie on the external face of $G$. Then the \FCN of $P$ is at most 4.
\end{corollary}

\section{Four forests are necessary}\label{sec:controesempio}

In this section we prove that in the general case a \NCS $P$ may satisfy $\PCFN(P)=4$. As a consequence, the result in Theorem~\ref{th:4_forest} is tight.

\begin{theorem}
There exists a \NCS $P$ such that $\PCFN(P)=4$.
\end{theorem}
\begin{proof}
The \NCS $P$ of this proof is built recursively and the proof is made step by step for convenience and readability.
\begin{enumerate}[label=\Alph*),leftmargin=0pt,itemindent=*,listparindent=\parindent,parsep=0pt]
\item\label{item:contr_a} For each $k$ there exists a \NCS $P_k$ such that:
\begin{itemize}\itemsep0em
\item $P_k=\Touch_{p_1}$, where $p_1$ is the path corresponding to the root of the genealogy tree,
\item $T_g^{P_k}$ is a complete binary tree composed by $2^k-1$ paths,
\item if $f$ is a face of type I for $p$, then the lower boundary of $f$ has exactly two edges, i.e., the lower boundary consists of $e^f_r$ and $e^f_\ell$,
\item for each face $f$ of type I for $p_1$, it holds that $e^f_r\in q'$, for all $q'\in L(q^f_r)$, and $e^f_\ell\in q''$, for all $q''\in R(q^f_\ell)$ (we recall that $q^f_r$ and $q^f_\ell$ are the right child and the left child, respectively, of the upper path of $f$, see Definition~\ref{def:upper_path_e_simili}).
\end{itemize}
the construction of $P_k$ can be obtained by generalizing $P_3$ in Figure \ref{fig:G_3}. We consider $P_{13}$.

For each $w$ leaf of $P_6$ let $R(w)=\{r_1^w,\ldots,r_8^w\}$ be ordered so that $r_{i+1}^w\preceq r_i^w$, for all $i\in[7]$. Similarly, let $L(w)=\{\ell_1^w,\ldots,\ell_8^w\}$ be ordered so that $\ell_{i+1}^w\preceq \ell_i^w$, for all $i\in[7]$ (note that $r_1^w=\ell_1^w=w$). For all $w$ leaf of $P_6$, we add three paths $q_1^w,q_2^w,q_3^w$ whose extremal vertices are in the subpath of the infinite face between $y_{\ell_1^w}$ and $y_{\ell_8^w}$ according to Figure~\ref{fig:controesempio}. Similarly, we add three paths $p_1^w,p_2^w,p_3^w$ whose extremal vertices are in the subpath of the infinite face between $x_{r_1^w}$ and $x_{r_8^w}$ according to a symmetric version of Figure~\ref{fig:controesempio}. 

In this way we obtain a \NCS $P$  composed by $6\cdot 2^5+2^{13}-1=8383$ paths. Let us assume by contradiction  that there exists a forest labeling $\mathcal{L}:P\mapsto[3]$. We say that a face $f$ is \emph{unsolved by $\mathcal{L}$} if $\bigcap_{e\in E(\partial f)}\mathcal{L}(e)\neq\emptyset$. To finish the proof it suffices to prove that there exists a face $f$ in $\bigcup_{p\in P}p$ unsolved by $\mathcal{L}$.

\item\label{item:contr_b} There exist three paths $q,q',q''\in P_{13}$ such that $q\preceq q'\preceq q''$, $|\mathcal{L}(\{q,q',q''\})|=3$ and $q$ is a leaf of $P_4$.  %for every path labeling $\mathcal{L}':P_4\mapsto[3]$ there exist three paths $q,q',q''$ satisfying $q\preceq q'\preceq q''$ and $|C'(\{q,q',q''\})|=3$. Thus, in $P_13}$ let $q$ leaf of $\mathcal{4}$

\item\label{item:contr_c} Let $f,f_1,f_2$ be the faces in $P_{13}$ whose upper paths are $q,q_r,q_\ell$, respectively. If $e$ is an extremal edge of the lower boundary of a face $g$ in $\{f,f_1,f_2\}$ then $|\mathcal{L}(e)|<3$, otherwise $g$ would be unsolved by $\mathcal{L}$ because of~\ref{item:contr_b}.

%\item Dato questo cammino $q$, osservo la faccia $f$ che si forma con i figli di $q$ e le facce $f_1,f_2$ che si formano con i figli dei figli di $q$\todob{per capire meglio, le facce $f,f_1,f_2$ sono le uniche tre facce di Figura \ref{fig:G_3}}. Dato un extremal edge $e$ di $f,f_1,f_2$ deve valere $\mathcal{L}(e)<3$, altrimenti si crea un ciclo facciale per via dell'altro extremal edge e per il punto 2.

\item \label{item:contr_d}Let $q_{r,\ell}$ be the left child of $q_r$ and let $q_{\ell,r}$ be the right child of $q_\ell$. By~\ref{item:contr_c}, for some distinct $c_1,c_2\in[3]$ it happens
\begin{itemize}
\item $\mathcal{L}(R(q_{r,\ell}))=\{c_1\}$ and $\mathcal{L}(L(q_{r,\ell}))=\{c_1,c_2\}$, or
\item $\mathcal{L}(L(q_{r,\ell}))=\{c_1,c_2\}$ and $\mathcal{L}(L(q_{r,\ell}))=\{c_1\}$, or
\item $\mathcal{L}(R(q_{\ell,r}))=\{c_1\}$ and $\mathcal{L}(L(q_{\ell,r}))=\{c_1,c_2\}$, or
\item $\mathcal{L}(L(q_{\ell,r}))=\{c_1,c_2\}$ and $\mathcal{L}(L(q_{\ell,r}))=\{c_1\}$,
\end{itemize}

indeed, if no one of the previous one applies, then 

\begin{itemize}
\item $|\mathcal{L}(R(q_{r,\ell}))|=1$ and $|\mathcal{L}(L(q_{r,\ell}))|=1$, which in turn implies $\mathcal{L}(R(q_{r,\ell}))=\mathcal{L}(L(q_{r,\ell}))=\mathcal{L}(q_{r,\ell})$ and thus the face $f_2$ is unsolved by $\mathcal{L}$, absurdum, or
\item $|\mathcal{L}(R(q_{\ell,r}))|=1$ and $|\mathcal{L}(L(q_{\ell,r}))|=1$, similar to the previous case, or
\item $|\mathcal{L}(R(q_{r,\ell}))|=|\mathcal{L}(L(q_{r,\ell}))|=|\mathcal{L}(R(q_{r,\ell}))|=|\mathcal{L}(L(q_{r,\ell}))|=2$, and thus $f$ is unsolved by $\mathcal{L}$ because of~\ref{item:contr_b} and \ref{item:contr_c}, absurdum.
\end{itemize}

%\todo{devo assumere uno tra a,b,c,d, gli altri sono analoghi}

\item\label{item:contr_e} By~\ref{item:contr_d}, there exists a leaf $w$ of $P_6$ such that either $\mathcal{L}(R(w))=\{c\}$ and $\mathcal{L}(L(w))=\{c,c'\}$ for some distinct $c,c'\in[3]$, or $\mathcal{L}(L(w))=\{c\}$ and $\mathcal{L}(R(w))=\{c,c'\}$. Let us assume that the latter case applies.

\item\label{item:contr_f} Now, \ref{item:contr_a} and \ref{item:contr_b} imply that every path $w'$ in $P_{13}$ such that $w'\preceq w$ is labeled according to algorithm $\LC(p_1,w,( c,c',c''))$, where $c''=\{1,2,3\}\setminus\{c,c'\}$. Hence, $\mathcal{L}(\{\ell_2^w,\ell_4^w,\ell_6^w,\ell_8^w\})=\{c'\}$ and $\mathcal{L}(\{\ell_2^w,\ell_4^w,\ell_6^w,\ell_8^w\})=\{c\}$.

\item\label{item:contr_g} Let $g_1,g_2,g_3$ be the faces generated by paths $q_1^w,q_2^w,q_3^w$ as depicted in Figure~\ref{fig:controesempio}. By~\ref{item:contr_f}, if $\mathcal{L}(q_1^w)\in\{c,c'\}$, then $f_1$ is unsolved by $\mathcal{L}$, thus $\mathcal{L}(q_1^w)=c''$. Similarly, $\mathcal{L}(q_2^w)=c''$ and $\mathcal{L}(q_3^w)=c''$. Finally, $c''\in\bigcap_{e\in E(g_3)}\mathcal{L}(e)$, and thus $g_3$ is unsolved by $\mathcal{L}$, absurdum.\qedhere
\end{enumerate}
\end{proof}

\begin{figure}[h]
\captionsetup[subfigure]{justification=centering}
\centering
%FIGURA 1
	\begin{subfigure}{10cm}
\begin{overpic}[width=10cm,percent]{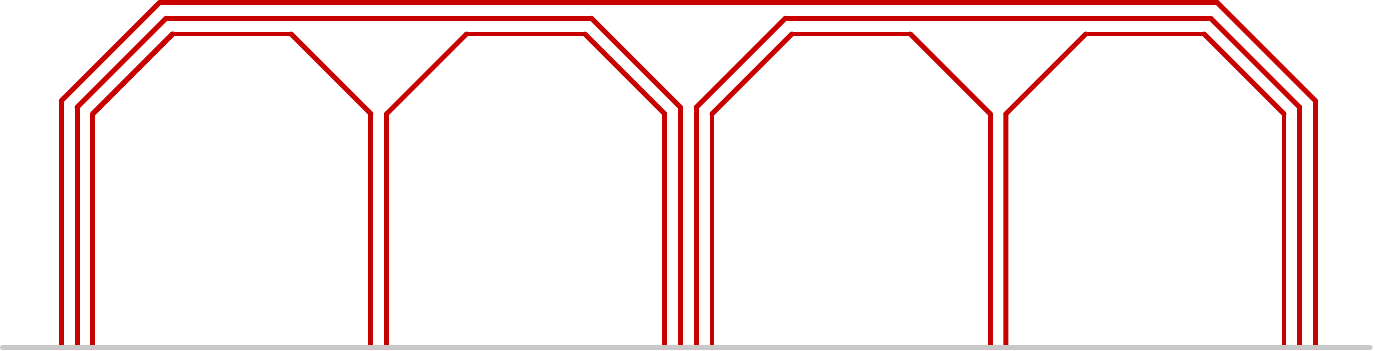}
\put(49.2,21){$f$}
\put(26,20.3){$f_2$}
\put(71.2,20.3){$f_1$}
\end{overpic}
%\caption{}\label{LABEL1}
\end{subfigure}
\caption{the \NCS $P_3$ and faces $f,f_1,f_2$.}
\label{fig:G_3}
\end{figure}

\begin{figure}[h]
\captionsetup[subfigure]{justification=centering}
\centering
%FIGURA 1
	\begin{subfigure}{6cm}
\begin{overpic}[width=6cm,percent]{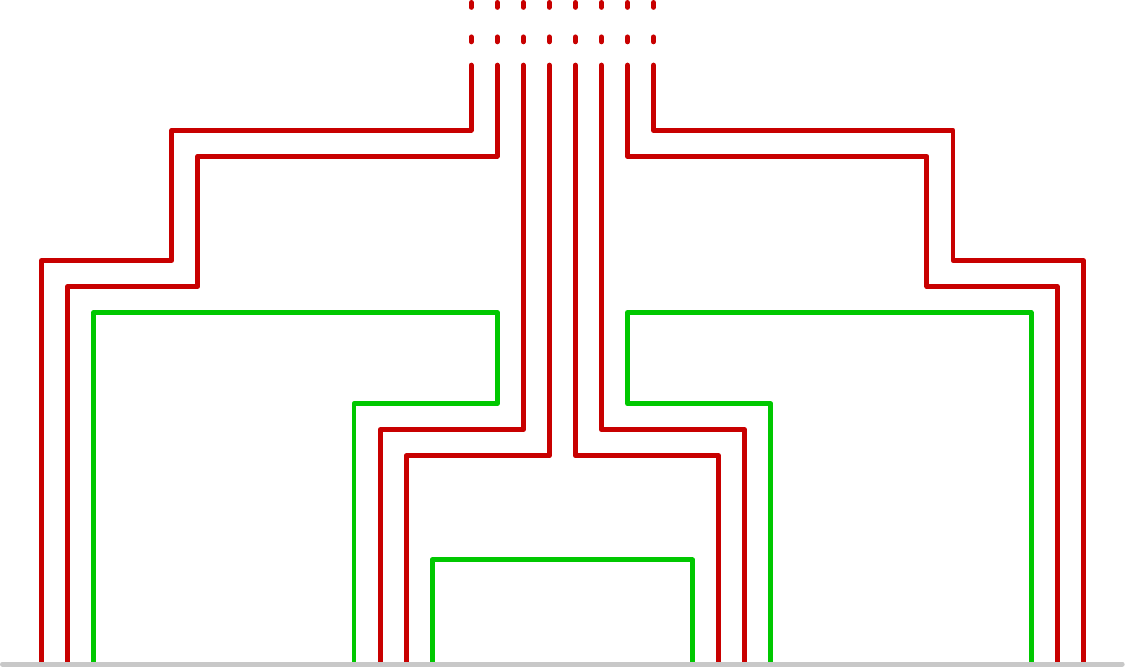}

\put(29,38){$g_1$}
\put(67,38){$g_2$}
\put(48,13){$g_3$}

\put(22,25.5){$q_1^w$}
\put(74,25.5){$q_2^w$}
\put(48,4){$q_3^w$}

\put(-6,0){$f^\infty$}

\end{overpic}
%\caption{}\label{LABEL1}
\end{subfigure}
\caption{paths $q_1^w,q_2^w$ and $q_3^w,$}
\label{fig:controesempio}
\end{figure}

\section{Conclusion}\label{sec:conclusions}

We introduced the \FCN, that is a variant of the classical covering problem. We showed that the \FCN is treatable for shortest paths in planar graphs. The main proved result states that if $P$ is a set of non-crossing shortest paths in a plane graph whose extremal vertices lie on the same face, then the \FCN of $P$ is at most 4; we also prove that this bound is tight.

We hope that more results on \FCN or its variants for particular graphs and paths classes could lead to more efficient algorithms for shortest paths and distance problems.

\section*{Acknowledgements}
We wish to thank Paolo G. Franciosa for deep discussion and his helpful advice.

\bibliography{biblio_tesi.bib}
\bibliographystyle{siam}

\end{document}